\documentclass[11pt,twoside,a4paper]{articlefrench}  
\ifx\pdfpageheight\undefined\PassOptionsToPackage{dvips}{graphicx}%
\else%
\PassOptionsToPackage{pdftex}{graphicx}
\PassOptionsToPackage{pdftex}{color}
\fi
\usepackage{amsfonts}
\usepackage{array}

\usepackage[french]{babel}

\ifx\pdfpageheight\undefined
\else
\usepackage[bookmarksopen=false,pdftex=true,breaklinks=true,%
    backref=page,pagebackref=true,plainpages=false,%
    hyperindex=true,pdfstartview=FitH,colorlinks=true,%
    pdfpagelabels=true,colorlinks=true,linkcolor=blue,citecolor=red]%
	{hyperref}
\fi

\title{Modules projectifs de type fini, \alis \cros et inverses g\'en\'eralis\'es}

\author{ Gema M. Diaz--Toca$^{*}$\\
Dpto. de Matematicas Aplicada\\
Universidad de Murcia, Spain\\
{\tt gemadiaz@um.es}
\and
Laureano Gonzalez--Vega\thanks{Partially 
supported by MCyT grant BFM 2002-04402-C02-0.}\\
       Dpto. de Matem\'aticas \\
Univ. of Cantabria, Spain \\  {\tt laureano.gonzalez@unican.es}
\and
Henri Lombardi \\ \'Equipe de Math\'ematiques, UMR CNRS 6623
\\ Univ. de Franche-Comt\'e, France
\\ {\tt henri.lombardi@univ-fcomte.fr}
\and
Claude Quitt\'e \\
Laboratoire de Math\'ematiques,
SP2MI, \\
Universit\'e de Poitiers, France\\
{\tt quitte@math.univ-poitiers.fr}}

\medskip 
\date{ 2004 \\
(version plus d\'etaill\'ee de l'article paru au Journal of Algebra {\bf 303} (2006) 450--475)}

\DeclareSymbolFont{lasy}{U}{lasy}{m}{n}
\SetSymbolFont{lasy}{bold}{U}{lasy}{b}{n}
\let\Box\undefined
\DeclareMathSymbol\Box{\mathord}{lasy}{"32}

\newtheorem{prop}{Proposition}[section]
\newtheorem{theorem}[prop]{Th\'eor\`eme}
\newtheorem{thdef}[prop]{Th\'eor\`eme et d\'efinition}
\newtheorem{proposition}[prop]{Proposition}

\newtheorem{corollary}[prop]{Corollaire}

\newtheorem{lemma}[prop]{Lemme}
\newtheorem{fact}[prop]{Fait}

\newtheorem{definition}[prop]{D\'efinition}

\newtheorem{notation}[prop]{Notation}

\newenvironment{proof}[1]{
\trivlist \item[\hskip \labelsep{\bf #1}]}{\hfill\mbox{$\Box$}
\endtrivlist}

\newcommand \num {{n$^{{\rm o}}$}}

\DeclareRobustCommand{\guig}{\mbox{{\usefont{U}{lasy}%
{\if b\expandafter\@car\f@series\@nil b\else m\fi}{n}%
\char40\kern-0.20em\char40}~}}
\DeclareRobustCommand{\guid}{\mbox{~\usefont{U}{lasy}%
{\if b\expandafter\@car\f@series\@nil b\else m\fi}{n}%
\char41\kern-0.20em\char41}}
\newcommand\gui[1]{\guig{#1}\guid}
\newcommand\junk[1]{}

\newcommand\bul{^\bullet}
\newcommand\cir{^\circ}

\newcommand\wt{\widetilde}

\renewcommand\ss{\smallskip}
\newcommand\ms{\medskip}

\newcommand\sni{\ss \noindent}

\newcommand\noi{\noindent}

\newcommand\qed{\hfill$\sqcap\kern-8.0pt\hbox{$\sqcup$}$}

\newcommand\gen[1]{\left\langle{#1}\right\rangle}
\newcommand\aqo[2]{#1\left/\!\gen{#2}\right.}

\newcommand\tra[1]{{\,^{\rm t}\!#1}}

\newcommand\cmatrix[1]{\left[\matrix{#1}\right]}

\newcommand \rC[1]{{\rm C}_{#1}}
\newcommand \rP[1]{{\rm P}_{\!#1}}
\newcommand \rR[1]{{\rm R}_{#1}}
\newcommand \rd[1]{{\rm d}_{#1}}

\newcommand\Tr{\mathrm{Tr}}

\renewcommand\det{\mathrm{det}}
\newcommand\Adj{\mathrm{Adj}}
\newcommand\Rad{\mathrm{Rad}}
\newcommand\rg{\mathrm{rg}}

\newcommand\Ker{\mathrm{Ker}}
\newcommand\Coker{\mathrm{Coker}}
\renewcommand\Im{\mathrm{Im}}
\newcommand\Id{\mathrm{Id}}

\newcommand\I{\mathrm{I}}
\newcommand\Ig{\mathrm{Ig}}
\newcommand\In{\I_n}

\newcommand \NN{\mathbb{N}}

\newcommand \RR{\mathbb{R}}

\newcommand \QQ{\mathbb{Q}}

\newcommand \CC{\mathbb{C}}

\newcommand\A{\mathbf{A}}

\newcommand\gB{\mathbf{B}}

\newcommand\cC{\mathcal{C}}
\newcommand\cD{\mathcal{D}}
\newcommand\cF{\mathcal{F}}

\newcommand\cG {\mathcal{G}}
\newcommand\cI {\mathcal{I}}
\newcommand\cL{\mathcal{L}}

\newcommand\cO{\mathcal{O}}
\newcommand\cP{\mathcal{P}}
\newcommand\cU{\mathcal{U}}

\newcommand \cad{c'est-\`a-dire }
\newcommand \cade{c'est-\`a-dire en\-co\-re }

\newcommand \ssi {si et seu\-le\-ment si }

\newcommand \Propeq {Les pro\-pri\-\'e\-t\'es sui\-van\-tes sont 
\'equi\-val\-en\-tes~:}

\newcommand\algo{algo\-rithme }  
\newcommand\algos{algo\-rithmes }  
\newcommand\algoz{algo\-rithme}

\newcommand \ali {appli\-ca\-tion lin\'e\-aire }
\newcommand \alis {appli\-ca\-tions lin\'e\-aires }
\newcommand \alisz {appli\-ca\-tions lin\'e\-aires}

\newcommand\Amo{$\A$--\,mo\-du\-le }  
\newcommand\Amos{$\A$--\,mo\-du\-les }  
  
\newcommand\Amosz{$\A$--\,mo\-du\-les}  

\newcommand \arith{arith\-m\'e\-ti\-que }  
  
\newcommand \ariths{arith\-m\'e\-ti\-ques }  
\newcommand \arithsz{arith\-m\'e\-ti\-ques}

\newcommand \auto {au\-to\-mor\-phis\-me }

\newcommand \cara{ca\-rac\-t\'e\-ris\-ti\-que }  
  
\newcommand \caraz{ca\-rac\-t\'e\-ris\-ti\-que}

\newcommand \coe{coef\-fi\-cient }  
\newcommand \coes{coef\-fi\-cients }

\newcommand \coli {com\-bi\-nai\-son \lin }

\newcommand \comz {co\-ma\-xi\-maux}

\newcommand \cro {croi\-s\'ee }
\newcommand \cros {croi\-s\'ees }

\newcommand \crosz {croi\-s\'ees}

\newcommand \deter{d\'e\-ter\-mi\-nant }  
\newcommand \deterz{d\'e\-ter\-mi\-nant}

\newcommand \dfn{d\'e\-fi\-ni\-tion }

\newcommand \egt {\'e\-ga\-li\-t\'e }
\newcommand \egts {\'e\-ga\-li\-t\'es }
\newcommand \egtz {\'e\-ga\-li\-t\'e}

\newcommand \elr{\'e\-l\'e\-men\-tai\-re }  
\newcommand \elrs{\'e\-l\'e\-men\-tai\-res }  
\newcommand \elrz{\'e\-l\'e\-men\-tai\-re}  
\newcommand \elrsz{\'e\-l\'e\-men\-tai\-res}  

\newcommand \elt{\'e\-l\'e\-ment }  
\newcommand \elts{\'e\-l\'e\-ments }  
  
\newcommand \eltsz{\'e\-l\'e\-ments}

\def \endo {en\-do\-mor\-phis\-me }
\def \endos {en\-do\-mor\-phis\-mes }
\def \endoz {en\-do\-mor\-phis\-me}

\newcommand\evcs{es\-pa\-ces vec\-to\-riels }

\newcommand\gne{g\'e\-n\'e\-ra\-li\-s\'e }  
  
\newcommand\gnes{g\'e\-n\'e\-ra\-li\-s\'es }  
\newcommand\gnees{g\'e\-n\'e\-ra\-li\-s\'ees }  
\newcommand\gnez{g\'e\-n\'e\-ra\-li\-s\'e}  
  
\newcommand\gnesz{g\'e\-n\'e\-ra\-li\-s\'es}

\newcommand\gnle{g\'e\-n\'e\-ra\-le }

\newcommand\gnlt{g\'e\-n\'e\-ra\-le\-ment }

\newcommand\gnn{g\'e\-n\'e\-ra\-li\-sa\-tion }  
\newcommand\gnns{g\'e\-n\'e\-ra\-li\-sa\-tions }

\newcommand\gtr{g\'e\-n\'e\-ra\-teur }  
\newcommand\gtrs{g\'e\-n\'e\-ra\-teurs }

\newcommand \id {id\'e\-al }
\newcommand \ids {id\'e\-aux }
\newcommand \idz {id\'e\-al}

\newcommand \ida {iden\-ti\-t\'e al\-g\'e\-bri\-que }
\newcommand \idas {iden\-ti\-t\'es al\-g\'e\-bri\-ques }

\newcommand \idaz {iden\-ti\-t\'e al\-g\'e\-bri\-que}

\newcommand \idc  {iden\-ti\-t\'e de Cramer }
\newcommand \idcs {iden\-ti\-t\'es de Cramer }

\newcommand \idd  {id\'eal d\'eter\-mi\-nantiel }

\newcommand \idds {id\'eaux d\'eter\-mi\-nantiels }

\newcommand \idfs {id\'eaux de Fitting }

\newcommand \idm {idem\-po\-tent }
\newcommand \idms {idem\-po\-tents }
\newcommand \idmz {idem\-po\-tent}
\newcommand \idmsz {idem\-po\-tents}

\newcommand \idt {iden\-tit\'e }
\newcommand \idts {iden\-tit\'es }

\newcommand \idtr {ind\'e\-ter\-mi\-n\'ee }

\newcommand \idtrz {ind\'e\-ter\-mi\-n\'ee}

\newcommand \ing {inver\-se \gne }
\newcommand \ings {inver\-ses \gnes }
\newcommand \ingz {inver\-se \gnez}
\newcommand \ingsz {inver\-ses \gnesz}

\newcommand \iMPz {inver\-se de Moo\-re-Pen\-ro\-se}

\newcommand \iMPsz {inver\-ses de Moo\-re-Pen\-ro\-se}

\newcommand \iso {iso\-mor\-phis\-me }
\newcommand \isos {iso\-mor\-phis\-mes }
\newcommand \isosz {iso\-mor\-phis\-mes}
\newcommand \isoz {iso\-mor\-phis\-me}

\newcommand \itf {\id \tf}
\newcommand \itfs {\ids \tf}
\newcommand \itfz {\id \tfz}

\newcommand \lin {lin\'e\-aire }
\newcommand \lins {lin\'e\-aires }

\newcommand \lnl {loca\-le\-ment \nl}
\newcommand \lnls {loca\-le\-ment \nls}
\newcommand \lnlz {loca\-le\-ment \nlz}

\newcommand \mpf {modu\-le \pf}

\newcommand \mptf {modu\-le \ptf}
\newcommand \mptfs {modu\-les \ptfs}
\newcommand \mptfsz {modu\-les \ptfsz}
\newcommand \mptfz {modu\-le \ptfz}

\newcommand \mrcs {modu\-les de rang cons\-tant }

\newcommand \mtfsz {modu\-les \tfz}

\newcommand \nl {sim\-ple }
\newcommand \nlz {sim\-ple}
\newcommand \nls {sim\-ples }
\newcommand \nlsz {sim\-ples}

\newcommand\op{op\'e\-ra\-tion }  
\newcommand\ops{op\'e\-ra\-tions }

\newcommand\opari{\op\arith}  
\newcommand\oparis{\ops\ariths}  
  
\newcommand\oparisz{\ops\arithsz}

\newcommand\orts{or\-tho\-go\-naux }  
\newcommand\ortes{or\-tho\-go\-na\-les }

\newcommand\ortsz{or\-tho\-go\-naux}

\newcommand\paral{pa\-ral\-l\`e\-le }  
\newcommand\parals{pa\-ral\-l\`e\-les }  
\newcommand\paralz{pa\-ral\-l\`e\-le}  
\newcommand\paralsz{pa\-ral\-l\`e\-les}

\newcommand\paralm{pa\-ral\-l\`e\-le\-ment }

\newcommand\pb{pro\-bl\`e\-me }

\newcommand \pf {de pr\'e\-sen\-ta\-tion finie }
\newcommand \pfz {de pr\'e\-sen\-ta\-tion finie}

\newcommand\pol{po\-ly\-n\^o\-me }  
  
\newcommand\pols{po\-ly\-n\^o\-mes }  
\newcommand\polsz{po\-ly\-n\^o\-mes}

\newcommand\polcar{\pol \cara}  
\newcommand\polcarz{\pol \caraz}

\newcommand \prn {pro\-jec\-tion }
\newcommand \prns {pro\-jec\-tions }

\newcommand \prnsz {pro\-jec\-tions}

\newcommand \pro {pro\-jec\-tif }
\newcommand \pros {pro\-jec\-tifs }

\newcommand \ptf {\pro \tf}
\newcommand \ptfs {\pros \tf}
\newcommand \ptfsz {\pros \tfz}
\newcommand \ptfz {\pro \tfz}

\newcommand \qnl {quasi-\nl}

\newcommand \qnlz {quasi-\nlz}

\newcommand \sli {sys\-t\`e\-me li\-n\'e\-ai\-re }
\newcommand \slis {sys\-t\`e\-mes li\-n\'e\-ai\-res }
\newcommand \slisz {sys\-t\`e\-mes li\-n\'e\-ai\-res}

\newcommand \tf {de type fini }
\newcommand \tfz {de type fini} 

\newcommand \Tho {Th\'eo\-r\`e\-me }
\newcommand \Thos {Th\'eo\-r\`e\-mes }
\newcommand \tho {th\'eo\-r\`e\-me }
\newcommand \thos {th\'eo\-r\`e\-mes }
\newcommand \thoz {th\'eo\-r\`e\-me}

\newcommand \cof {cons\-truc\-tif }

\newcommand \cov {cons\-truc\-ti\-ve }

\newcommand \covsz {cons\-truc\-ti\-ves}

\newcommand \comaz {\maths\covsz}

\def \cot {cons\-truc\-ti\-ve\-ment }

\newcommand \maths {ma\-th\'e\-ma\-ti\-ques }

\newcommand \prco {preuve \cov}

\begin{document}   

\maketitle
\begin{abstract}  
D'une part, nous d\'eveloppons la th\'eorie g\'en\'erale des \ings
 de matrices en la mettant en rapport avec la th\'eorie constructive des
\mptfsz. D'autre part nous pr\'ecisons certains aspects de cette th\'eorie li\'es 
au calcul formel et \`a l'analyse num\'erique matricielle. Nous d\'emontrons en 
particulier qu'on peut tester si un
\Amo \pf est projectif et calculer une matrice de projection correspondante 
\gui{en temps polynomial}. Plus pr\'ecis\'ement pour une
matrice $A\in \A^{m \times n}$ on peut d\'ecider s'il existe un \ing $B$ pour $A$ 
(\cad une matrice $B$ v\'erifiant $ABA=A$ et $BAB=B$) et, en cas de r\'eponse 
positive, calculer un tel \ing par un \algo qui utilise 
$\cO(p^6\,q^{2})$ \oparis (avec $p=\inf(m,n)$, $q=\sup(m,n)$) et un nombre 
polynomial de tests d'appartenance d'un \elt \`a un id\'eal engendr\'e par \gui{un 
petit nombre d'\eltsz.}. 
\end{abstract}
\section{Introduction}       
Dans cet article  $\A$ d\'esigne un anneau commutatif arbitraire.
D'une part, nous d\'eveloppons la th\'eorie g\'en\'erale des \ings
 de matrices  en la mettant en rapport avec la th\'eorie constructive des
\mptfsz. D'autre part nous pr\'ecisons certains aspects de cette th\'eorie li\'es 
au calcul formel et \`a l'analyse num\'erique matricielle.

Nous utiliserons une mesure assez grossi\`ere de la complexit\'e des calculs sur 
machine:
cette complexit\'e sera mesur\'ee essentiellement \`a travers le nombre d'\oparis 
de base dans $\A$.

Nous supposerons en outre souvent qu'il y a sur l'anneau $\A$ un test explicite 
d'appartenance \`a un \itf 
(l'anneau est \gui{fortement discret}
selon la terminologie des \comaz).
Par exemple un corps explicite est fortement discret \ssi il poss\`ede un test 
d'\'egalit\'e \`a z\'ero.
Nous supposerons aussi que ce test 
pour \gui{$x\in\gen{x_1,\ldots ,x_n}~?$} (avec la r\'eponse compl\`ete en cas 
d'appartenance) utilise un nombre d'op\'erations \gui{\elrsz}
born\'e par $\cO(n^{s})$ (nous ne pr\'ecisons pas plus la nature exacte de ces 
op\'erations). 
Nous dirons alors que $\A$  est $\cO(n^{s})$-fortement discret.
Notez que le test \`a z\'ero utilise donc un nombre d'\ops \elrs
born\'e par une constante.

Dans la suite une \gui{op\'eration \elrz} sera ou bien une \opari de base dans 
l'anneau, ou bien l'une des op\'erations \'el\'ementaires qui interviennent dans 
le test d'appartenance \`a un \itfz.

Par exemple on a facilement.
\markboth{Introduction}{Introduction}
\begin{lemma} 
\label{lemIdIdm} 
Sur un anneau  $\cO(n^{s})$-fortement discret, on a un test pour d\'eterminer si 
un \itf $\gen{x_1,\ldots ,x_n}$
est \idm et donner, en cas de r\'eponse positive un \gtr \idm de l'\idz.
Ce test utilise un nombre d'\oparis en $\cO(n^4)$ et un nombre d'autres \ops \elrs 
en  $\cO(n^{2s+1})$.
\end{lemma}
\begin{proof}{Preuve}
R\'esulte imm\'ediatement du \gui{d\'eterminant trick} qui prouve qu'un
\itf \idm est engendr\'e par un \idmz. On a besoin du r\'esultat des
 $n$ tests d'appartenance \gui{$x_i\in\gen{x_1,\ldots ,x_n}^2\,?$}. Le  $\cO(n^4)$ 
\oparis provient
du calcul du \deter qui fournit l'\idm recherch\'e. 
\end{proof}

Un \sli sur  $\A,$ pr\'esent\'e sous forme matricielle $AX=Y$ 
($A\in\A^{m{\times}n}$), est parti\-cu\-li\`erement \gui{agr\'eable} si on peut 
calculer
une solution (quand il en existe une) en fonction lin\'eaire de $Y$,
autrement dit, quand il existe une matrice $B\in\A^{n{\times}m}$  
telle que $A\,B\,A\,X=A\,X$ pour tout $X$, i.e. $A\,B\,A=A$.
Dans le cas o\`u ceci est possible, nous disons que l'\ali d\'efinie
par $A$ est \emph{\lnlz}. Si en outre
$B\,A\,B=B$ la matrice $B$ est appel\'ee un \emph{\ingz} de~$A$.

La litt\'erature sur le sujet des \ings est assez consid\'erable.
Nous renvoyons plus particuli\`erement \`a \cite{BIG},  \cite{Bha}, \cite{BP}, 
\cite{Lan} ou \cite{RM}.

Pour ce qui concerne les \mptfs qui donnent pour l'essentiel la m\^{e}me th\'eorie 
sous une forme un peu plus abstraite, nous renvoyons \`a \cite{Nor} et pour un 
traitement \elr et \cof \`a \cite{LQ}. 

\ms Nous citons maintenant quelques r\'esultats significatifs obtenus dans le 
travail pr\'esent.

Nous devons d'abord introduire (ou rappeler) quelques d\'efinitions.

Soient $E$ et $F$ deux \Amosz.  Deux \alis $\varphi:E\rightarrow F$ et 
$\varphi\bul:F\rightarrow E$ sont dites \emph{\crosz} si on a: 
\begin{equation} \label{eqSDO}
  \Im\,\varphi \oplus \Ker\,\varphi\bul= F\,, \; 
\; \Ker\,\varphi\oplus \Im\,\varphi\bul= E\, 
\end{equation}

Nous notons $Q_n$ la matrice diagonale ayant pour \coe en position $(k,k)$ la 
puissance $t^{k-1}$ o\`u $t$ est une \idtrz. Si $A\in\A^{m{\times}n}$ on note 
$A\cir$ la matrice ${Q_m}^{-1}\; A\; Q_n$. 

L'anneau $\A(t)$ est le localis\'e $S^{-1}\A[t]$ o\`u $S$ est l'ensemble des \pols 
primitifs (i.e., les \coes engendrent l'id\'eal $\gen{1}$).

Certains des \'enonc\'es qui suivent sont un peu moins pr\'ecis que dans le texte.

Les deux premiers \thos que nous citons doivent sans doute se trouver
dans la litt\'erature. Du moins il est raisonnable de penser qu'ils font partie du 
folklore. 


\medskip\noindent
{\bf \Thos \ref{thCHam} et \ref{corthIgRangr2}}
{\it  Soient $E$ un \Amo \ptfz,  $\varphi:E\rightarrow E$ une \ali et 
$\rP{\varphi}(Z)=\det(\,\Id_E + Z\varphi)=1+\sum_{\ell\geq 1} d_\ell\,Z^\ell$. 
\Propeq
\begin{enumerate}
\item  $\varphi$ est crois\'ee avec elle-m\^{e}me, et $\Im\,\varphi$ 
est un module projectif de rang $k$.
\item $\varphi$ est de rang $\leq k$  et $d_k$ est inversible.
\item $\deg\,\rP{\varphi}\leq k$, 
$d_ k$ est inversible et, en d\'efinissant $\pi$ par

\sni\centerline{$
\pi=d_{k-1}\varphi-d_{k-2}\varphi ^2+
\cdots +(-1)^{k-1}\varphi^k,\quad \quad ~$}

\sni on a les \egts $\pi\,\varphi=d_ k\,\varphi$ et $\pi^2=d_k\,\pi$.
\end{enumerate}
}

Les \thos qui suivent sont, \`a notre connaissance, nouveaux.


\medskip\noindent
{\bf \Thos \ref{thIgRangr} et \ref{corthIgRangr}}
{\it   
Soient $E$ et $F$ deux \Amos \ptfs et deux \alis $\varphi:E\rightarrow F$ et 
$\varphi\bul:F\rightarrow E$. Posons  $\rP{\varphi \varphi\bul}(Z)=\det(\,\Id_F + 
Z\varphi \varphi\bul)=1+ \sum_{\ell\geq 1} a_\ell\, Z^\ell$. \Propeq
\begin{enumerate}
\item  $\varphi$ et $\varphi\bul$ sont crois\'ees et $\Im\,\varphi$ 
est un module projectif de rang $k$.
\item  $\varphi$ et $\varphi\bul$ sont de rang $\leq k$ et  $a_k$ est inversible.
\item   $\deg\,\rP{\varphi \varphi\bul}\leq k$, $a_k$ est inversible et, en 
d\'efinissant $\theta$ par

\sni\centerline{$
\theta =a_{k-1}\, \varphi \bul
-a_{k-2}\,\varphi \bul\varphi \varphi \bul+
\cdots +(-1)^{k-1}\varphi \bul(\varphi \varphi \bul)^{k-1}$,}

\sni on a les deux \egts $\varphi\,\theta\,\varphi=a_k\,\varphi$ et 
$\varphi\bul\,\varphi \,\theta =a_k\,\varphi\bul$.
\end{enumerate}
}


\medskip\noindent
{\bf \Tho \ref{thDecision1}}
{\it  
Soient $E$ et $F$ deux \Amos \ptfs engendr\'es par $n$ \elts (ou moins), et deux 
\alis $\varphi:E\rightarrow F$ et $\varphi\bul:F\rightarrow E$. Alors on peut, 
avec un nombre
d'\oparis en $\cO(n^4)$, et un nombre de tests \gui{$x\in\gen{y}\,?$} en 
$\cO(n^3)$, d\'ecider si $\varphi$ et $\varphi\bul$ sont \crosz, et en cas de 
r\'eponse positive calculer des \ings  de $\varphi$ et $\varphi\bul$ en $\cO(n^4)$ 
\oparisz.
}

\junk{\medskip Deux variantes (\thos \ref{thDecisionO} et \ref{thDecision1bis})
donnent des versions de l'\algo pr\'ec\'edent en $\cO(\log^2n)$ \'etapes
\paralsz.}


\medskip\noindent
{\bf \Tho \ref{thIGGRC2}}
{\it 
Soit  une matrice $A\in\A^{m{\times}n}$. On pose $P(Z,t)=\det(\I_n+ZAA\cir)=1+ 
\sum_{\ell\geq 1}g_\ell(t)Z^\ell$.
Les propri\'et\'es suivantes sont \'equivalentes:
\begin{enumerate}
\item $A$ est \lnl  de rang $k$ sur $\A$.
\item $A$ est \lnl  de rang $k$ sur $\A(t)$. 
\item $A$ et $A\cir$ sont \cros sur $\A(t)$, de rang $k$.
\item $A$ et $A\cir$ sont \cros sur $\A(t)$,
$\deg_Z(P)\leq k$ et 
le \pol $t^{k(n-k)}g_k(t)$ est primitif.  
\item  $\deg_Z(P)\leq k$,
le \pol $t^{k(n-k)}g_k(t)$ est primitif et si on pose
$B(t)=g_{k-1}(t)A\cir-g_{k-2}(t)A\cir A A\cir  + \cdots+
(-1)^{k-1}(A \cir A )^{k-1}A \cir$, on a $A\cdot B(t)\cdot A=g_k(t)\,A$.
\end{enumerate}
Si $\A$ est un anneau r\'eduit, la derni\`ere condition  se simplifie en 
\gui{$\deg_Z(P)\leq k$ et 
le \pol $t^{k(n-k)}g_k(t)$ est primitif}.
Lorsque les conditions sont v\'erifi\'ees  $B(t)/g_k(t)$ est un \ing de
$A$ sur l'anneau $\A(t)$.
 }


\medskip\noindent
{\bf \Tho \ref{thIGGRC2bis}}
{\it 
Soit  une matrice $A\in\A^{m{\times}n}$.
Les propri\'et\'es suivantes sont \'equivalentes:
\begin{enumerate}
\item $A$ est \lnl sur $\A$. 
\item $A$ est \lnl sur $\A(t)$. 
\item $A$ et $A\cir$ sont \cros sur $\A(t)$.  
\end{enumerate}
 }


\medskip\noindent
{\bf \Tho \ref{thIGGG}}
{\it Sur un anneau $\A$ fortement discret,
on peut tester si une matrice $A\in\A^{m{\times}n}$ est \lnlz, et en cas de 
r\'eponse positive, calculer un inverse \gne
de la matrice.  Soit $p=\min(m,n)$, $q=\max(m,n)$.
Si l'anneau est  $\cO(n^{s})$-fortement discret, ces calculs consomment 
$\cO(p^6\,q^{2}+p\,q^{4})$ \oparis et 
$\cO(p^4\,q+pq^{2s+1})$ autres \ops \elrsz.
Avec les m\^{e}mes bornes de complexit\'e, on calcule un \ing de $A$ et des 
matrices de
\prn sur le noyau et sur l'image de $A$.
 }

\medskip 
Dans nos calculs de complexit\'e, nous avons utilis\'e les \algos de 
multiplication usuels pour les \pols et les matrices.
On peut donc am\'eliorer les performances en utilisant des
\algos de multiplication rapide.

\medskip 
Signalons enfin que les preuves de cet article
reposent en partie sur des \idts de Cramer \gnees 
(voir sections  \ref{subsecIdtCra1} et \ref{subsecInter})
dont nous avons eu du mal \`a trouver la trace dans la litt\'erature.
Nous remercions \`a ce sujet d'une part les statisticiens indiens et d'autre part 
Mustapha Rais pour un expos\'e \`a Poitiers dans lequel il interpr\'etait les 
r\'esultats de \cite{DiGL} au moyen de la th\'eorie des invariants.
\section{Identit\'es de Cramer et premier \ing} 
\label{subsecIdtCra}

\subsection{Formules de Cramer usuelles et inusuelles}    
\label{subsecIdtCra1}
Une matrice $A\in\A^{m{\times}n}$ sera dite \emph{de rang $\leq k$}
si tous les mineurs d'ordre $k+1$ sont nuls.
Pour une matrice $A\in\A^{m{\times}n}$ nous noterons 
$A_{\alpha,\beta}$ la matrice extraite sur les lignes  
$\alpha=\{\alpha_1,\ldots ,\alpha_r\}\subset\{1,\ldots,m\}$ et les colonnes  
$\beta=\{\beta_1,\ldots ,\beta _s\}\subset\{1,\ldots,n\}$.

Si $B$ est une matrice carr\'ee d'ordre $n$, nous notons $\wt{B}$ ou $\Adj\,B$ la 
matrice cotranspos\'ee (on dit parfois adjointe). La forme \elr des
\idcs s'\'ecrit alors $B\,\wt{B}=\wt{B}\,B=\det\,B\;\I_n$.

Supposons la matrice $A$ de rang $\leq k$. Soit $V\in\A^{m{\times}1}$ un vecteur 
colonne tel que  $(A\,|\,V)$ soit aussi de rang $\leq k$.
Appelons $A_j$ la $j$-\,\`eme colonne de $A$.
Soit $\mu_{\alpha,\beta}=\det(A_{\alpha,\beta})$ le mineur 
d'ordre $k$ de la matrice $A$ extrait sur les lignes  
$\alpha=\{\alpha_1,\ldots ,\alpha_k\}$ et les colonnes  
$\beta=\{\beta_1,\ldots ,\beta _k\}$. 
Pour $j=1,\ldots ,k$ soit $\nu_{\alpha,\beta,j}$ le 
\deter
de la m\^{e}me matrice extraite, \`a ceci pr\`es que la 
colonne $j$
a \'et\'e remplac\'ee par la colonne extraite de $V$  
sur les lignes   $\alpha$.
Alors on obtient pour chaque couple $(\alpha,\beta)$ de
multi-indices et chaque $j\in\{1,\ldots ,k\}$ une \idt de Cramer:
\begin{equation} \label{eqMPC1}
\qquad \mu_{\alpha,\beta}\;V=\sum\nolimits_{j=1}^k 
\nu_{\alpha,\beta,j}\,A_{\beta_j}\qquad  
\end{equation}
due au fait que le rang de la matrice 
$(A_{1..m,\beta}\,|\,V)$ est  
$\leq k$. Ceci peut se relire comme 
suit:
$$\begin{array}{rcl} 
\qquad \mu_{\alpha,\beta}\;V&=& \left[
\begin{array}{ccc}
    A_{\beta_1} & \ldots  & A_{\beta_k}
\end{array} 
\right] \,  \left[
\begin{array}{c}
      \nu_{\alpha,\beta,1} \\
    \vdots  \\ 
     \nu_{\alpha,\beta,k}
\end{array}
\right]=\\[10mm]
 &=&\left[
\begin{array}{ccc}
    A_{\beta_1} & \ldots  & A_{\beta_k}
\end{array} 
\right] \,
\Adj(A_{\alpha,\beta})
 \,
 \left[
 \begin{array}{c}
     v_{\alpha_1}  \\
     \vdots   \\
     v_{\alpha_k}
 \end{array} \right]=\\[10mm]
 &=&A \, (\I_n)_{1.. n,\beta} \,
\Adj(A_{\alpha,\beta}) \, (\I_m)_{\alpha,1.. m}\, V
\end{array}$$

Ceci nous conduit \`a introduire la notation suivante
\begin{notation} 
\label{notaAdjalbe}
{\rm Nous notons $\cP_{k,\ell}$ l'ensemble des parties \`a $k$ \elts de
$\{1,\ldots ,\ell\}$. Pour  $A\in\A^{m{\times}n}$ et $\alpha\in 
\cP_{k,m},\,\beta\in \cP_{k,n}$ nous notons  
$$\Adj_{\alpha,\beta}(A):=
(\I_n)_{1.. n,\beta} \,
\Adj(A_{\alpha,\beta}) \, (\I_m)_{\alpha,1.. m}\,.$$
} 
\end{notation}

L'\'egalit\'e pr\'ec\'edente s'\'ecrit alors:
\begin{equation} \label{eqGema}
\mu_{\alpha,\beta}\;V\;=\;A \, \Adj_{\alpha,\beta}(A) \, V
\end{equation}

Comme cons\'equence on obtient, toujours sous l'hypoth\`ese que $A$ est de rang 
$\leq k$:
\begin{equation} \label{eqIGCram}
\mu_{\alpha,\beta}\;A\;=\;A \, \Adj_{\alpha,\beta}(A) \, A
\end{equation}

Voici un exemple de l'\egt $\mu_{\alpha,\beta}\;V\;=\;A \, \Adj_{\alpha,\beta}(A) 
\, V$ pour voir la matrice $\Adj_{\alpha,\beta}(A)$. Supposons que nous avons
le syst\`eme lin\'eaire:
$$\left[\begin{array}{ccc} 5&-5&7\\
\noalign{\medskip}9&0&2\\\noalign{\medskip}13&5&-3\end {array}
\right]\, X=\left[\begin{array}{c} 26 \\ \noalign{\medskip} 6 \\
\noalign{\medskip} -14 \end {array} \right]
=\left[\begin{array}{c} v_1 \\ \noalign{\medskip} v_2 \\
\noalign{\medskip} v_3 \end {array} \right],$$
 avec
$\rg(A)=\rg(A\,|\,V)=2$. Prenons $\alpha=\{1,2\}$ et
$\beta=\{2,3\}$, alors:

$$ \mu_{\alpha,\beta}=\left| \begin {array}{cc} -5&7\\\noalign{\medskip}0&2\end
{array} \right|, \; \sigma_{\alpha,\beta,1}=\left| \begin
{array}{cc} 26 & 7 \\ \noalign{\medskip}6 &2\end {array} \right|,
\, \sigma_{\alpha,\beta,2}=\left| \begin {array}{cc}
-5&26\\\noalign{\medskip}0&6\end {array} \right|, \,
\Adj(A_{\alpha,\beta})=\left[ \begin {array}{cc}
2&-7\\\noalign{\medskip}0&-5\end {array} \right],
$$
$$
(\I_3)_{1..3,\beta}=\left[\begin{array}{cc}
    0 & 0 \\
    1 & 0 \\
    0 & 1
  \end{array}\right], \quad (\I_3)_{\alpha,1.. 3}=\left[\begin{array}{ccc}
    1& 0 & 0 \\
    0 & 1 & 0 \\
  \end{array}\right],\quad \Adj_{\alpha,\beta}(A)=\left[
   \begin {array}{ccc} 0&0&0\\\noalign{\medskip}2&-7&0\\
   \noalign{\medskip}0&-5&0\end {array} \right],
$$

et
  \begin{eqnarray*}
\mu_{\alpha,\beta}\;V &=& \sigma_{\alpha,\beta,1}\,A_2
+\sigma_{\alpha,\beta,2}\,A_3 = \left[
\begin{array}{cc}
     A_2 &  A_3
\end{array}
\right] \, \Adj(A_{\alpha,\beta})
  \,
  \left[
  \begin{array}{c}
      v_1  \\
      v_2
  \end{array} \right]=\\
  &=&A \, \left[\begin{array}{cc}
    0 & 0 \\
    1 & 0 \\
    0 & 1
  \end{array}\right] \,
\Adj(A_{\alpha,\beta}) \, \left[\begin{array}{ccc}
    1& 0 & 0 \\
    0 & 1 & 0 \\
  \end{array}\right]\, V
  =
A \, \Adj_{\alpha,\beta}(A) \, V
\end{eqnarray*}
\begin{definition} \label{defIdDet} 
Soit $A\in\A^{m\times n}$, les {\em  \idds de la matrice $A$}  sont les id\'eaux 

\ss\centerline{$\cD_k(A) \; 
:=  \;$  l'id\'eal   engendr\'e  par  les   mineurs d'ordre  $k$  de  la  matrice   
$A$
}

\sni o\`u $k$ est un entier arbitraire. 
Pour $k\leq 0$ les mineurs sont par convention \'egaux \`a 1, pour 
$k> \min(m,n)$ ils sont par convention \'egaux \`a $0$.
Si $A$ est la matrice d'une \ali $\varphi$ les \ids $\cD_k(A)$ ne d\'ependent que 
de $\varphi$ et sont donc aussi appel\'es 
{\em  \idds de l'\ali $\varphi$}.
\end{definition}

Les \idts de Cramer vues pr\'ec\'edemment fournissent des congruences qui ne sont 
soumises \`a aucune hypoth\`ese: il suffit par exemple de lire (\ref{eqGema}) dans 
l'anneau quotient $\A/\cD_{k+1}(A\,|\,V)$ pour obtenir la congruence 
(\ref{eqGema2}).

\begin{lemma} 
\label{lemCram}  Avec les notations pr\'ec\'edentes mais sans aucune hypoth\`ese 
sur la matrice $A$ ou le vecteur $V$ on a pour $\alpha\in \cP_{k,m},\;\beta\in 
\cP_{k,n}:$
\begin{eqnarray} \label{eqGema2}
\mu_{\alpha,\beta}\;V&\equiv&A \, \Adj_{\alpha,\beta}(A) \, V
\qquad \mathrm{mod}\quad \cD_{k+1}(A\,|\,V)
\\[1mm] \label{eqCGCram}
\mu_{\alpha,\beta}\;A&\equiv&A \, \Adj_{\alpha,\beta}(A) \, A
\qquad \mathrm{mod}\quad \cD_{k+1}(A)\,.
\end{eqnarray}
\end{lemma}

Une cons\'equence imm\'ediate de l'\idc (\ref{eqIGCram})
est l'\idt  suivante moins usuelle.

\begin{proposition} 
\label{propIGCram} 
Soit $A\in\A^{m\times n}$ de rang $\leq k$ avec $\cD_k(A)=\gen{1}$.
Pr\'ecis\'ement supposons 
$$\sum\nolimits_{\alpha\in \cP_{k,m},\beta\in \cP_{k,n}} 
c_{\alpha,\beta}\,\mu_{\alpha,\beta}=1 \quad \mathit{et \;posons}
\quad B\;=\;\sum\nolimits_{\alpha\in \cP_{k,m},\beta\in 
\cP_{k,n}}\,c_{\alpha,\beta}\,\Adj_{\alpha,\beta}(A)\,.$$
Alors 
\begin{equation} \label{eqIGCram2}
A\, B\, A=A\,.
\end{equation}
En cons\'equence $A\, B$ est une \prn et $\Im\,A=\Im\,AB$ 
est facteur direct dans~$\A^m$.   
\end{proposition}

L'\idt suivante est encore plus miraculeuse (voir \cite{Bha} \thoz~5.5).
\begin{proposition} 
\label{propIGCram2} \emph{(Prasad et Robinson)}
Avec  les  hypoth\`eses et les notations de la proposition
pr\'ec\'edente, si $\forall \alpha,\alpha'\in \cP_{k,m},$ $\forall\beta,\beta '\in 
\cP_{k,n}$ $\;c_{\alpha,\beta}\,c_{\alpha',\beta'}= 
c_{\alpha,\beta'}\,c_{\alpha',\beta}$,
alors 
\begin{equation} \label{eqIGCramPraRo}
B\, A\, B=B\,.
\end{equation}
\end{proposition}

\subsection{Applications lin\'eaires \nls et lemme de la libert\'e}  
Nous ne savons pas s'il existe une terminologie officielle pour la notion 
suivante. 
\begin{definition} 
\label{defAliNl} 
Une \ali $\varphi:E\rightarrow F$   entre deux \Amos libres de dimensions finies 
est dite \emph{\nl (de rang $k$)} si, pour des bases convenables $(e_1,\ldots 
,e_n)$ et $(f_1,\ldots ,f_m)$ de $E$ et $F$ on a: $\varphi(e_i)=f_i$ si $i\leq k$ 
et $\varphi(e_i)=0$ si $i> k.$ 
\end{definition}
Il revient au m\^eme de dire que $\Ker\,\varphi$ et $\Im\,\varphi$ sont libres et 
admettent des suppl\'ementaires libres. 
Ou encore que la matrice de $\varphi$ sur des bases arbitraires de $E$ et $F$ 
s'\'ecrit  $A=U\,\I_{k,m,n}\, V$  avec $U$ et $V$ inversibles, et 
$\I_{k,m,n}\in\A^{m{\times}n}$ est de la forme $\cmatrix{\I_k&0\cr 0& 0}.$

Si on pose $B=V^{-1}\,\I_{k,n,m}\, U^{-1}$ on a imm\'ediatement
$$A\, B\, A=A \quad  \mathrm{et} \quad  B\, A\, B=B\,.$$

L'\ali $x\mapsto ax$ de $\A$ dans $\A$ est \nl \ssi $a$ est nul ou inversible.

Le rang d'une \ali \nl est bien d\'efini d\`es que l'anneau n'est pas trivial. 
Avec l'anneau trivial par contre, toutes les \alis sont \nlsz, de tous rangs 
(cette remarque est n\'ecessaire pour admettre sans r\'eticence le lemme 
\ref{lemLib} ainsi que le point
\ref{h} du \tho \ref{theoremIFD}).

Le lemme suivant (voir \cite{LQ}) est imm\'ediat.
\begin{lemma} 
\label{lemLib} \emph{(lemme de la libert\'e)}
Soit $\varphi:E\rightarrow F$  une \ali de rang $\leq k$ entre deux \Amos libres 
de dimensions finies. Soit $A$ une matrice repr\'esentant $\varphi$ sur des bases 
de $E$ et $F$. Soit $\mu$ un mineur d'ordre $k$ 
de $A$. Si $\mu$ est inversible,  $\varphi$ est \nl de rang $k$. En particulier 
$\varphi$ est toujours \nl de rang $k$ sur l'anneau $\A[1/\mu]$.  
\end{lemma}

\subsection{Syst\`emes fondamentaux d'\idms \ortsz}  
Un syst\`eme fondamental d'\idms \orts (sfio) est une famille finie
$(r_i)_{1\leq i\leq n}$ qui v\'erifie $r_ir_j=0$ si $i\neq j$ et 
$\sum_{i=1}^nr_i=1$.
Il revient au m\^{e}me de se donner un tel syst\`eme dans $\A$ ou de se donner un 
\iso $\A\rightarrow \A_1\times \cdots \times \A_n$. 
L'\idm $r_i$ dans $\A$ correspond alors au \gui{vecteur} 
$(0,\ldots ,0,1,0,\ldots 0)$ avec $1$ en position $i$ dans 
$\A_1\times \cdots \times \A_n$.
 Chaque 
$\A_i$ est isomorphe \`a $\A[1/r_i]\simeq \aqo{\A}{1-r_i}$, ou encore
\`a l'\id $r_i\A$ qu'on consid\`ere comme un anneau unitaire en prenant $r_i$ 
comme \elt neutre pour la multiplication (attention, ce n'est pas un sous anneau 
de $\A$, parce que le neutre n'est pas le m\^{e}me).

Dans nos \'enonc\'es, nous ne supposerons pas que tous les $r_i$ dans un sfio sont 
non nuls. Cela nous simplifie la vie (et les \'enonc\'es) notamment lorsqu'on n'a 
pas de test d'\egt \`a $0$ dans l'anneau consid\'er\'e.  Il faut simplement se 
rappeler que l'anneau $\A[1/0]$ est trivial pour comprendre pourquoi les 
\'enonc\'es restent justes. 

Une \gnn naturelle de la notion d'\ali \nl lorsque l'anneau poss\`ede des \idms 
est la suivante.

\begin{definition} 
\label{defAliQnl} 
Une \ali $\varphi:E\rightarrow F$   entre deux \Amos libres de dimensions finies 
est dite \emph{\qnlz} si, pour des bases convenables $(e_1,\ldots ,e_n)$ et 
$(f_1,\ldots ,f_m)$ de $E$ et $F$ on a: $\varphi(e_i)=r_if_i$  $(1\leq i\leq 
\inf(m,n))$ o\`u les $r_i$ sont des \idms v\'erifiant $r_ir_{i+1}=r_{i+1}$ $(1\leq 
i< \inf(m,n))$, et si
et $\varphi(e_i)=0$ pour $i> \inf(m,n)$. \end{definition}

Posons $r_0=1$, $r_{\inf(m,n)+1}=0$ et 
$s_i=r_i-r_{i+1}$ ($0\leq i\leq \inf(m,n)$). Alors les $s_i$ forment un sfio et 
$\varphi$ devient \nl de rang $k$ lorsqu'on \'etend les scalaires \`a l'anneau 
$\A[1/s_k]$.

R\'eciproquement il est facile de voir qu'une \ali qui devient simple
chaque fois qu'on localise en les \elts d'un sfio est \qnlz.

\subsection{Inverses g\'en\'eralis\'es et \alis \crosz}  
\label{subsecIGALC}

Dans les sections suivantes, nous donnerons plusieurs \gnns du r\'esultat des 
propositions \ref{propIGCram} et \ref{propIGCram2}, qui nous donnent notre premier 
\gui{\ingz}.
La terminologie concernant les \ings ne semble pas enti\`erement fix\'ee. Nous 
adoptons celle de \cite{Lan}. Dans \cite{Bha} l'auteur utilise le terme 
\gui{reflexive g-inverse}:

\begin{definition} 
\label{defIng} 
Soient $E$ et $F$ deux \Amosz,  et une \ali $\varphi:E\rightarrow F$. Une \ali 
$\psi :F\rightarrow E$ est appel\'ee un \emph{\ingz} de $\varphi$ si on a $\varphi 
\circ\psi \circ\varphi =\varphi$ et $\psi \circ\varphi \circ\psi =\psi.$ 
\end{definition}

Dans ces conditions, on v\'erifie que $\varphi\, \psi$ et $\psi\, \varphi$ sont 
des \prnsz, que $\Im\,\varphi =\Im\,\varphi\, \psi$,  $\Im\,\psi  =\Im\,\psi\, 
\varphi$,  $\Ker\,\varphi =\Ker\,\psi \,\varphi$, $\Ker\,\psi  
=\Ker\,\varphi\,\psi$, et donc $E=\Ker\,\varphi \oplus \Im\,\psi$ et 
 $F=\Ker\,\psi  \oplus \Im\,\varphi$.

Si on a une \ali $\psi_1$ v\'erifiant $\varphi\, \psi_1\,\varphi =\varphi$ 
on obtient un
\ing en posant $\psi=\psi_1\,\varphi \,\psi_1$.

Le lemme suivant  d\'ecrit les \ings d'une \ali \nlz.

\begin{lemma} 
\label{lemIngNl} 
Soit $E$ et $F$ des modules libres de dimensions finies et $\varphi:E\rightarrow 
F$ une \ali \nl dont la matrice
sur des bases fix\'ees est $A=U\,\I_{r,m,n}\, V$ ($U$ et $V$ sont inversibles, cf. 
d\'efinition \ref{defAliNl}). Alors les \ings de $\varphi$ sont toutes les \alis 
ayant (sur les m\^{e}mes bases) une matrice $B\in\A^{n{\times}m}$ de la forme 
suivante (avec $C\in\A^{r{\times}(m-r)}$ et 
$D\in\A^{(n-r){\times}r}$):
$$ B= V^{-1}\, \cmatrix{\I_r&C\cr D& DC} \, U^{-1}
$$
\end{lemma}

%

Si $\psi$ est un \ing de $\varphi$, alors $\varphi$ et $\psi$ sont \crosz.

R\'eciproquement, la connaissance d'une \ali \cro avec $\varphi$ permet de 
calculer un \ing de $\varphi$.
En effet si $\varphi \bul$ est \cro avec $\varphi$,  $\varphi$ 
se restreint en un \iso $\varphi_0$ de 
$\Im\,\varphi\bul$ sur $\Im\,\varphi$ et 
$\varphi\bul$ se restreint en un \iso 
$\varphi_0\bul$ de $\Im\,\varphi$ sur 
$\Im\,\varphi\bul$. On~a:
\begin{equation} \label{SDO2}
\begin{array}{rclrrcl}
\Im\,\varphi  & =  &\Im\,\varphi\varphi\bul\,,   & \quad&
\Ker\,\varphi\bul  & = & \Ker\,\varphi\varphi\bul\,, 
\\[1mm]
\Ker\,\varphi   & =  & \Ker\,\varphi\bul\varphi\,,  &
&\Im\,\varphi\bul & = & \Im\,\varphi\bul\varphi\,.
\end{array}
\end{equation}

Notons $\pi_{\Im\,\varphi}:F\rightarrow F$ la \prn sur $\Im\,\varphi$ \paralm \`a 
$\Ker\,\varphi\bul$. On d\'efinit l'\ali $\psi:F\rightarrow E$ par
\begin{equation} \label{eqIg1}
\forall y \in F\qquad  
\psi (y)= \varphi_0^{-1}(\pi_{\Im\,\varphi}(y)).
\end{equation}

Il est alors clair que $\psi$ convient comme \gui{\ing de $\varphi$ via 
$\varphi\bul$} au sens du \tho suivant. 
\begin{thdef} 
\label{thCrIvg} 
Si $\varphi:E\to F$ et $\varphi\bul:F\to E$ sont \cros 
il existe une
unique \ali $\psi :F\rightarrow E$ v\'erifiant les deux conditions:
\begin{enumerate}
\item  $\varphi \circ \psi$  est la \prn sur $\Im\,\varphi$ \paralm
\`a $\Ker\,\varphi\bul$;
\item $\psi  \circ \varphi $  est la \prn sur $\Im\,\varphi\bul$ \paralm \`a 
$\Ker\,\varphi.$
\end{enumerate}
Cette \ali $\psi$ peut \^{e}tre aussi caract\'eris\'ee par les 4 \egts suivantes:
\begin{equation} \label{eqIg1.1}
\varphi \circ \psi\circ \varphi =\varphi,\qquad  
\psi \circ \varphi \circ \psi =\psi,\qquad
\varphi\bul \circ \varphi \circ \psi  =\varphi\bul,\qquad
 \psi\circ  \varphi \circ \varphi\bul =\varphi\bul\,.
\end{equation}
Nous dirons que \emph{$\psi$ est l'\ing de $\varphi$ via $\varphi\bul$} et nous le 
noterons 
$\psi=\Ig(\varphi,\varphi\bul)=\varphi^{\dag_{\varphi\bul}}$.
\end{thdef}
\begin{proof}{Preuve}
Il nous reste  \`a voir que les quatre \egts suffisent.\\
Puisque $\varphi\, \psi\,\varphi =\varphi$ et $ \psi\,\varphi\,\psi  =\psi $ on a 
$E=\Ker\,\varphi \oplus\Im\,\psi $, $F=\Ker\,\psi  \oplus\Im\,\varphi$, $\varphi\, 
\psi$  est la \prn sur $\Im\,\varphi$ \paralm \`a $\Ker\,\psi $ et $\psi  \, 
\varphi $  est la \prn sur $\Im\,\psi $ \paralm \`a $\Ker\,\varphi.$
  Il nous suffit donc de montrer que $\Ker\,\varphi \bul=\Ker\,\psi$ et
$\Im\,\varphi \bul=\Im\,\psi$.\\
La troisi\`eme \egt implique  $\Ker\,\psi\subset \Ker\,\varphi \bul$. On conclut  
$\Ker\,\varphi \bul=\Ker\,\psi$ en remarquant que
$F=\Ker\,\psi  \oplus\Im\,\varphi= \Ker\,\varphi\bul\oplus\Im\,\varphi$.\\
De m\^{e}me la derni\`ere \egt implique $\Im\,\varphi \bul \subset\Im\,\psi$ et on 
conclut de la m\^{e}me fa\c{c}on.
\end{proof}

Le \tho pr\'ec\'edent correspond \`a la \dfn donn\'ee par Moore, dans le cas 
d'\evcs hermitiens, avec pour $\varphi \bul$ la conjugu\'ee $\varphi^{*},$ ce qui 
donne des \prns \ortes et l'\iMPz.

\begin{lemma} 
\label{lemCroCro} 
Si $\varphi$ et $\varphi\bul$ sont crois\'ees, alors $\varphi \varphi \bul$ est 
crois\'ee avec elle-m\^{e}me (m\^{e}me chose pour $\varphi \bul \varphi$). En 
outre si $\theta=\Ig(\varphi \varphi \bul,\varphi \varphi \bul)$ alors 
$\varphi\varphi\bul \theta = \theta\varphi\varphi\bul$, 
$\theta=\Ig(\varphi\bul,\varphi)\,\Ig(\varphi ,\varphi\bul ),$ $\varphi\bul 
\theta=\Ig(\varphi ,\varphi\bul)$  et
$\theta\varphi=\Ig(\varphi\bul,\varphi)$.
\end{lemma}

On a une caract\'erisation purement \'equationnelle de la situation du \tho 
\ref{thCrIvg}, \`a condition d'introduire les deux \ingsz.
\begin{proposition} 
\label{thCrIvg2} 
Soient $E$ et $F$ deux \Amos et deux \alis $\varphi:E\rightarrow F$ et 
$\varphi\bul:F\rightarrow E$.
\begin{enumerate}
\item Si $\varphi$ et $\varphi\bul$ sont \crosz, posons 
$\psi=\Ig(\varphi,\varphi\bul)$ et $\psi\bul=\Ig(\varphi\bul,\varphi)$. On a:
\begin{equation} \label{eqCrIvg2}
\begin{array}{rclrclrcl}
\varphi \circ \psi\circ \varphi&=&\varphi 
&\qquad \varphi\bul \circ \psi\bul\circ\varphi\bul&=\varphi\bul
&\qquad \varphi \circ\psi &=&\psi \bul\circ\varphi\bul\\[1mm]
\psi \circ \varphi \circ \psi&=&\psi 
&\qquad\psi\bul \circ \varphi\bul \circ \psi\bul&=\psi\bul
&\qquad\psi \circ\varphi  &=&\varphi\bul \circ\psi \bul
\end{array}
\end{equation}
\item R\'eciproquement si $\psi$ et $\psi \bul$ v\'erifient les \'egalit\'es 
(\ref{eqCrIvg2}), alors $\varphi$ et $\varphi\bul$ sont \crosz, 
$\psi=\Ig(\varphi,\varphi\bul)$  et $\psi\bul=\Ig(\varphi\bul,\varphi)$.
\end{enumerate}
\end{proposition}

Un cas particulier est le suivant:
\begin{proposition} 
\label{thCrIvg2bis} 
Soit $E$ un \Amo et une \ali $\varphi:E\rightarrow E$.
\begin{enumerate}
\item Si $\varphi$ est \cro avec elle-m\^{e}me, posons 
$\psi=\Ig(\varphi,\varphi)$. On a:
\begin{equation} \label{eqCrIvg2bis}
\begin{array}{rclrclrcl}
\varphi \circ \psi\circ \varphi&=&\varphi 
&\qquad \psi \circ \varphi \circ \psi&=&\psi
&\qquad \varphi \circ\psi &=&\psi \circ\varphi
\end{array}
\end{equation}
\item R\'eciproquement si $\psi$ v\'erife les \'egalit\'es (\ref{eqCrIvg2bis}), 
alors $\varphi$ est \cro avec elle-m\^{e}me et $\psi=\Ig(\varphi,\varphi)$.
\end{enumerate}
\end{proposition}

Dans \cite{Bha}, lorsque sont v\'erifi\'ees les \'egalit\'es (\ref{eqCrIvg2bis}), 
$\psi$ est appel\'e un \gui{group inverse} de $\varphi$.

\subsection{Le cas des modules \tfz}

Nous d\'eveloppons maintenant un petit peu d'alg\`ebre lin\'eaire
sur les \mtfsz, en donnant quelques r\'esultats
bien connus pour les \evcs de dimension finie qui g\'en\'eralisent de mani\`ere 
parfois inattendue.

\begin{proposition} 
\label{prop quot non iso} {\em  (\cite{MRR} chap. III, exo. 9 p. 80)}
Soit $E$ un \Amo \tf et $\varphi~:E\rightarrow E$ une \ali surjective. Alors 
$\varphi$ est un \isoz. 
\end{proposition}

\junk{
\begin{proof}{Preuve}
Soit $(x_1,\ldots ,x_n)$ un syst\`eme g\'en\'erateur de $E$. Soit 
$\gB=\A[\varphi]\subset{\rm End}_\A(E)$ et $\cI=(\varphi)$ l'id\'eal principal de 
$\gB$ engendr\'e par $\varphi$. L'anneau $\gB$ est commutatif et on peut 
consid\'erer $E$ comme un $\gB$-module. Puisque $\varphi$ est surjective, il 
existe une matrice $P\in \cI^{n\times n}$ v\'erifiant 
$P \tra{(x_1,\ldots ,x_n)}=\tra{(x_1,\ldots ,x_n)}$, c.-\`a-d.
$(\In-P) \tra{(x_1,\ldots ,x_n)}=  \tra{(0,\ldots ,0)}~$ (o\`u $\In=(\In)_\gB$ est 
la matrice \idt de $\gB^{n\times n}$),
donc:

\smallskip \centerline{$ \det(\In-P)\tra{(x_1,\ldots ,x_n)}= 
\Adj (\In-P)\,(\In-P)\tra{(x_1,\ldots ,x_n)}= \tra{(0,\ldots ,0)}.$}
 
\sni Donc $\det(\In-P)=0_\gB$,  or $\det(\In-P)=1_\gB-\varphi \psi$ avec 
$\psi\in \gB$ (puisque $P$ est \`a \coes dans $\cI=\varphi \gB$). Ainsi $\varphi 
\psi=\psi \varphi =1_\gB=\Id_E$, donc $\varphi$ 
est inversible dans $\gB$.
\end{proof}
}

\begin{proposition} 
\label{propAutocro} Soit $E$ un \Amo \tf et une \ali $\varphi:E\rightarrow E$. 
\Propeq
\begin{enumerate}
\item $E=\Im\,\varphi \oplus \Ker\,\varphi$ (i.e. $\varphi$ est
\cro avec elle-m\^{e}me).
\item $E=\Im\,\varphi + \Ker\,\varphi$ 
\item $\Im\,\varphi =\Im\,\varphi^2$.
\end{enumerate}
\end{proposition}
\begin{proof}{Preuve}
1 implique clairement 2 et 3.\\
2 implique 3: Tout $x\in E$ s'\'ecrit $x=\varphi(y)+z$ avec $\varphi(z)=0$ donc 
tout $\varphi(x)\in\Im\,\varphi$ s'\'ecrit
$\varphi^2(y)$.\\
3 implique 2: Si $\varphi (x)=\varphi (\varphi (y))$ alors
$\varphi (x-\varphi (y))=0$ donc $x=\varphi(y)+z$ avec $\varphi (z)=0$.\\
2 et 3 impliquent 1: Soit $\varphi_0:\Im\,\varphi \rightarrow \Im\,\varphi$ 
obtenue par restriction de $\varphi$. Le module
$\Im\,\varphi$ est \tf puisque $E$ est \tfz. Mais $\varphi_0$ est surjective par 
hypoth\`ese. Donc, par la proposition \ref{prop quot non iso}
 $\varphi_0$ est bijective. Ceci implique clairement $\Ker\,\varphi \cap
\Im\,\varphi =0$.
\end{proof}

De la m\^{e}me fa\c{c}on:

\begin{proposition} 
\label{propCrocro} 
Soient $E$ et $F$ deux \Amos \tfz.  Des \alis $\varphi:E\rightarrow F$ et 
$\varphi\bul:F\rightarrow E$ telles que $\Im\,\varphi\bul+\Ker\,\varphi =E$ 
et  $\Im\,\varphi+\Ker\,\varphi\bul =F$ sont \crosz. 
\end{proposition}
\begin{proof}{Preuve}
Si $\Im\,\varphi\bul+\Ker\,\varphi =E$ alors $\Im\,\varphi \simeq E/\Ker\,\varphi 
\simeq \Im\,\varphi\bul/(\Ker\,\varphi\cap\Im\,\varphi\bul)$. 
De mani\`ere sym\'etrique $\Im\,\varphi\bul$ est isomorphe
\`a un quotient de $\Im\,\varphi$. En composant ces deux \isos
on trouve que $\Im\,\varphi$ est isomorphe \`a un 
quotient de lui-m\^{e}me par un sous-module plus grand que 
$\Ker\,\varphi\cap\Im\,\varphi\bul$. La proposition  \ref{prop quot non iso} 
implique donc que $\Ker\,\varphi\cap\Im\,\varphi\bul=0$. 
M\^{e}me chose pour  $\Ker\,\varphi\bul\cap\Im\,\varphi$.
\end{proof}

\section{Interpr\'etation de l'\ing avec des \idts de Cramer}   
\label{subsecInter}

Lorsqu'on a une matrice carr\'ee $A$ d'ordre $n$, il y a deux mani\`eres tr\`es 
diff\'erentes de
calculer sa matrice cotranspos\'ee $\Adj\,A$.
La premi\`ere consiste \`a calculer ses \coes qui sont, au signe pr\`es,
des mineurs d'ordre $n-1$ de $A$. La seconde consiste \`a utiliser le \tho de 
Cayley-Hamilton qui nous fournit un \pol $Q(X)$, facilement d\'eduit du \polcarz,
v\'erifiant $AQ(A)=\det\,A\;\I_n$. Alors $\Adj\,A=Q(A)$. 
Cette co\"{\i}ncidence peut \^{e}tre vue comme une famille d'\idas remarquables.
Dans cette section nous g\'en\'eralisons ce r\'esultat \gui{en rang $k<n$}.
Les choses sont cependant un peu plus d\'elicates et il est
plus pratique de travailler
avec deux \alisz.

\medskip Dans cette section $E$ et $F$ sont des modules libres de dimensions 
finies.
On consid\`ere deux \alis $\varphi:E\rightarrow F $ et
 $\varphi\bul:F\rightarrow E$. 
On ne suppose pas a priori que $\varphi$ et $\varphi\bul$ sont \crosz.
Soient $A$ et $A\bul$ des matrices pour $\varphi$ et $\varphi\bul$ sur des bases 
fix\'ees de $E$ et $F$.
On note pour simplifier $\mu_{\alpha ,\beta}=\det(A_{\alpha,\beta})$ 
et  $\mu\bul_{\beta,\alpha}=\det(A\bul_{\beta,\alpha}).$ 

La formule de  Binet-Cauchy montre que:
\begin{lemma} 
\label{propGramCroise}
Si $p$ est la plus petite des dimensions de  $E$ et $F$ et si
$\det(\Id_E+Z\,\varphi\bul \varphi)= 1+ a_1 Z + \cdots+ a_p Z^p$ alors, pour tout 
$k\leq p$:
$$  
a_k=\sum\nolimits_{\alpha\in \cP_{k,m},\beta\in 
\cP_{k,n}}\,\mu\bul_{\beta,\alpha}\,\mu_{\alpha,\beta}\,.
$$
\end{lemma}

\begin{notation} 
\label{notaGenAdj}
On reprend les hypoth\`eses pr\'ec\'edentes. 
On notera  $\cG^{(k)}_{\varphi\bul}(\varphi)=\rd k(\varphi\bul \varphi)=a_k$. Nous 
les appelerons des \emph{\coes de Gram mixtes}. Enfin
nous d\'efinissons $\Adj_{\varphi\bul,k}(\varphi)$ et 
 $\Adj^{(k)}_{\varphi\bul}(\varphi)$ par:
\begin{eqnarray} 
\label{eqDefAdjficirc}
\Adj_{\varphi\bul,k}(\varphi)&=&\sum\nolimits_{\alpha\in \cP_{k,m},\beta\in 
\cP_{k,n}}\mu\bul_{\beta,\alpha}\,\Adj_{\alpha,\beta}(\varphi )\\[2mm]
\label{eqDefAdjficirc2}
\Adj^{(k)}_{\varphi\bul}(\varphi)&=&
a_{k-1}\,\varphi \bul-a_{k-2}\,\varphi \bul\varphi\,\varphi \bul  + \cdots+
(-1)^{k-1}(\varphi \bul\varphi )^{k-1}\varphi \bul\,.
\end{eqnarray}
\end{notation}

A priori l'\ali que nous avons not\'ee $\Adj_{\varphi\bul,k}(\varphi)$  d\'epend 
du choix des bases de $E$ et $F$. Nous allons voir bient\^{o}t qu'il n'en est 
rien. En effet nous allons montrer:

\begin{theorem} 
\label{thGENCRAM3} 
On a toujours: 
\begin{equation} \label{eqGENCRAM3}
\Adj^{(k)}_{\varphi\bul}(\varphi) 
\;= \; \Adj_{\varphi\bul,k}(\varphi) 
\end{equation}
En fait ces \alis sont aussi \'egales au \emph{gradient} de la fonction
$\varphi \mapsto \rd k(\varphi\bul\varphi)$ (notez bien que $\varphi\bul$ est ici 
une constante).
\end{theorem}

Pour pr\'eciser la derni\`ere phrase, nous devons donner la d\'efinition du 
gradient d'une fonction polynomiale $\cL(E,F)\rightarrow \A$ (\cad une fonction 
qui est donn\'ee par un polynome en les entr\'ees de la matrice $A$ de $\varphi 
\in\cL(E,F)$ une fois choisies des bases de $E$ et $F$). 
Il ne s'agit de rien d'autre que la diff\'erentielle de la fonction, traduite sous 
forme d'un \elt $\theta\in\cL(F,E)$ en utilisant la dualit\'e canonique entre 
$\cL(E,F)$ et $\cL(F,E)$ donn\'ee par la forme bilin\'eaire \gui{trace du 
produit}. 

\begin{definition} 
\label{defGrad} 
Soit $a:\cL(E,F)\rightarrow \A$ une fonction polynomiale.
On appelle gradient de $a$ au point $\varphi$ et on note $\nabla(a)(\varphi)$ 
l'unique \ali $\theta\in\cL(F,E)$ telle que 
$a(\varphi +\epsilon )=a(\varphi)+ \Tr(\theta\,\epsilon) + \cO(\epsilon^{(2)})$,
o\`u  $\cO(\epsilon^{(2)})$ d\'esigne sous forme abr\'eg\'ee une fonction 
polynomiale de $\epsilon$ sans terme constant ni terme du premier degr\'e.
\end{definition}

\begin{proof}{Preuve du \tho \ref{thGENCRAM3}}
On utilise un fait bien connu et deux lemmes qui s'en d\'eduisent simplement. Nous 
donnons les preuves pour faciliter la lecture de l'article.
\begin{fact} 
\label{factAdjGrad} 
Pour un \endo  $\psi$  d'un \Amo $F$ libre de rang $m$ 
on a $$\nabla(\det)(\psi)=\Adj(\psi)\,.$$ 
\end{fact}
\begin{proof}{Preuve du fait \ref{factAdjGrad}}~\\
Voici une premi\`ere preuve. Raisonnons avec des matrices carr\'ees. On a 
$\det(\I_m+H)=1+\Tr\,H+\cO(H^{(2)})$ o\`u  $\cO(H^{(2)})$ est comme dans la 
d\'efinition \ref{defGrad}. Donc $\nabla(\det)(\I_m)=\I_m$.
Si $A$ est inversible on a  
$\det(A+H)=\det(A)\det(\I_m+A^{-1}H)=\det(A)+\det(A)\Tr\,A^{-1}H+ \cO(H^{(2)})$. 
Comme $\det(A)\Tr\,A^{-1}H=\Tr(\det(A)A^{-1}H)$ et
$\Adj\,A=\det(A)A^{-1}$ cela donne $\nabla(\det)(A)=\Adj\,A$.
On conclut en remarquant qu'on vient de d\'emontrer,
sous la condition \gui{$A$ inversible} une \ida dans laquelle on n'a pas 
pr\'ecis\'e le contenu exact du terme  $\cO(H^{(2)})$. Mais puisqu'il s'agit bien 
d'une \idaz, il suffisait de la d\'emontrer pour $A$ dans un ouvert de 
$\QQ^{n\times n}$.\\
Une autre preuve est la suivante: si $A_i$ (resp. $H_i$) d\'esigne la
$i$-\`eme colonne de $A$ (resp. de $H$), il est clair que la diff\'erentielle de 
$\det$ au point $A$ est l'\ali
$$ H\mapsto \det(H_1, A_2,\ldots , A_n) + \cdots  + \det(A_1, \ldots , A_{n-1}, 
H_n)\,.
$$
Par ailleurs l'\egt 
$$ \det(H_1, A_2, \ldots , A_n) + \cdots  + \det(A_1, \ldots , 
A_{n-1},H_n) = \Tr(\Adj(A)\,H)
$$
r\'esulte clairement des \idcs (cf. par exemple notre section 
\ref{subsecIdtCra1}). 
\end{proof}
Un premier corollaire imm\'ediat est le lemme suivant.
\begin{lemma} 
\label{lemAdjGrad} 
On fixe des bases de $E$ et $F$. Le gradient de la fonction $\cL(E,F)\rightarrow 
\A\,:\,\varphi\mapsto \mu_{\alpha ,\beta }=\det(A_{\alpha,\beta})$ (o\`u $A$  est 
la matrice de $\varphi$) au point $\varphi$ est l'\endo \gui{cotranspos\'e en 
$(\alpha,\beta)$} ayant pour  matrice
$\Adj_{\alpha,\beta}(A)$.
\end{lemma}
\begin{proof}{Preuve du lemme \ref{lemAdjGrad}}~\\
Raisonnons avec des matrices. Notons $J_\beta=(\I_n)_{1.. n,\beta}$ et 
$P_\alpha=(\I_m)_{\alpha,1.. m}$ les deux matrices telles que $P_\alpha \,A\, 
J_\beta=A_{\alpha,\beta}$.
Puisque l'application $\lambda : A\mapsto A_{\alpha,\beta}$ est lin\'eaire, la 
diff\'erentielle de $A\mapsto\det\,A_{\alpha,\beta}$ calcul\'ee au point $A$ pour 
l'accroissement $H$ est donn\'ee par
$$ \Tr(\Adj(A_{\alpha,\beta}) (\lambda(H)))=
\Tr(\Adj(A_{\alpha,\beta}) \,P_\alpha \,H\, J_\beta) =
\Tr(J_\beta\,\Adj(A_{\alpha,\beta})\, P_\alpha\, H ) =
\Tr(\Adj_{\alpha,\beta}(A) \, H )\,.
$$
Autrement dit
$$ \nabla(M\mapsto\det\,M_{\alpha,\beta})(A)=\Adj_{\alpha,\beta}(A)\,.
$$
\end{proof}
Vu le lemme \ref{propGramCroise}, un corollaire de ce lemme est que l'\ali 
$\Adj_{\varphi\bul,k}(\varphi)$ est le gradient de la fonction $\varphi \mapsto 
\rd k(\varphi\bul\varphi)$. En particulier, malgr\'e les apparences de sa 
d\'efinition, cette \ali ne d\'epend que de $\varphi,$ $\varphi\bul$ et $k$, et 
non des bases choisies.\\
L'autre lemme, bien connu en th\'eorie des invariants 
(voir par exemple \cite{Rai,Tu1,Tu2}), est:
\begin{lemma} 
\label{lemAdjGrad2} 
Pour un \endo  $\psi$  d'un \Amo libre $F$  on a 
$$\nabla(\rd k)(\psi)=\rd {k-1}(\psi)\,\Id_F-\rd {k-2}(\psi)\,\psi +\rd {k-
3}(\psi)\,\psi^2  + \cdots+
(-1)^{k-1}\psi^{k-1}\,.$$ 
\end{lemma}
\begin{proof}{Preuve du lemme \ref{lemAdjGrad2}}~\\
On pose $\gB=\cL(F,F)$ et on identifie $\gB[X]$ avec $\cL(F[X],F[X])$. 
On consid\`ere $\gB$ comme une $\A$-alg\`ebre et  $\gB[X]$ comme une $\A[X]$-
alg\`ebre. Nous d\'erivons la fonction 
$$\delta \,:\,\gB\rightarrow \A[X]\,:\, \psi \mapsto \det(1_\gB+X\psi)=1_\A+\rd 
1(\psi)\,X+\cdots +\rd m(\psi)\,X^m\,.
$$
Cette fonction est obtenue en composant la
fonction affine 
$$\gB\rightarrow\gB[X]\,:\,\psi \mapsto 1_\gB+X\psi
$$ 
et la fonction $\det:\gB[X]\rightarrow \A[X]$. 
Calculons cette diff\'erentielle au point $\psi$ pour un accroissement $\epsilon$. 
Nous obtenons l'\ali
$$\epsilon \mapsto \Tr(\Adj(1_\gB+X\psi)\,(X\,\epsilon))=
\Tr((X\,\Adj(1_\gB+X\psi))\,\epsilon)\,.
$$
L'\ali $\eta(X)=\Adj(1_\gB+X\psi)=1_\gB+\eta_1\,X+\cdots +\eta_{m-1} \,X^{m-1}$ 
(avec les $\eta_i\in\gB$) est donc \'egale \`a
$$ 1_\gB+\nabla(\rd2)(\psi)\,X+\cdots +\nabla(\rd{m})(\psi) \,X^{m-1}\,.
$$
Elle v\'erifie  $$\eta(X)\,(1_\gB+X\psi)=\det(1_\gB+X\psi)\,1_\gB=(1_\A+\rd 
1\,X+\cdots +\rd m\,X^m)\,1_\gB\,.$$
 On obtiendra donc
$\eta $ comme \elt de $\gB[X]$ en faisant dans  $\gB[X]$ la division
par puissances croissantes du \pol $1_\gB + \rd1\,X+\cdots +\rd n\,X^m\,$ par
$1_\gB+X\psi$ (le fait que la division est exacte, i.e., le reste est nul, fournit 
l'une des preuves usuelles du \tho de Cayley-Hamilton).
Et cela donne le r\'esultat annonc\'e.
\end{proof}
On d\'eduit enfin du lemme \ref{lemAdjGrad2} 
que l'\ali $\Adj^{(k)}_{\varphi\bul}(\varphi)$
est le gradient de la fonction 
$\varphi \mapsto \rd k(\varphi\bul\varphi)$. 
Nous devons en effet d\'eriver la fonction obtenue en composant la
fonction lin\'eaire $\varphi \mapsto \varphi\bul\varphi$ et la fonction $\rd k$: 
le gradient correspondant est bien $(\nabla(\rd k)(\varphi\bul \varphi 
))\,\varphi\bul$.
\end{proof}

Le \tho qui suit nous sera particuli\`erement utile dans la section
\ref{subsecCalPra}.

\begin{theorem} 
\label{thGENCRAM1} 
On a avec les notations pr\'ec\'edentes si $\varphi$ est de rang $\leq k$: 
\begin{eqnarray} \label{eqGENCRAM1}
\varphi \circ \Adj^{(k)}_{\varphi\bul}(\varphi) \circ \varphi& = & a_k\,\varphi .
\end{eqnarray}
\end{theorem}
\begin{proof}{Preuve} 
Une cons\'equence imm\'ediate de l'\idt de Cramer (\ref{eqCGCram}) est:
\begin{eqnarray} \label{eqAdj1}
\varphi \circ \Adj_{\varphi\bul,k}(\varphi) \circ \varphi& \equiv & a_k\,\varphi 
\qquad \mathrm{mod}\quad \cD_{k+1}(\varphi).
\end{eqnarray}
On conclut par le \tho \ref{thGENCRAM3}.
\end{proof}

Un cas particulier de la formule (\ref{eqIg3}) qui suit est la formule 2.13 dans 
\cite{PB}.
La signification est qu'un \ing d'une \ali $\varphi$ 
calcul\'e en utilisant une
\ali $\varphi\bul$ crois\'ee avec $\varphi$ donne la solution du \sli 
correspondant $AX=V$ ($A$ est la matrice de $\varphi$) sous forme d'une moyenne 
pond\'er\'ee d'\idcs du type (\ref{eqGema}) page~\pageref{eqGema}.

\begin{theorem} 
\label{thGENCRAM2} 
Si $\varphi$ et $\varphi\bul$ sont \cros  de rang  $k$, alors
\begin{equation} \label{eqIg3}
\Ig(\varphi,\varphi\bul)=a_k^{-1}\,\Adj_{\varphi\bul,k}(\varphi)
\end{equation}
\end{theorem}

\begin{proof}{Preuve}
Si  $\varphi$ et $\varphi\bul$ sont \cros de rang $k$ on obtient en appliquant la 
formule~(\ref{eqIg2}) page~\pageref{eqIg2} l'\egt 
$$\Ig(\varphi,\varphi\bul)=a_k^{-1}\, \Adj^{(k)}_{\varphi\bul}(\varphi).$$ 
On conclut par le \tho \ref{thGENCRAM3}.
\end{proof}

\medskip 
Notons cependant que pour le calcul de l'\ing ce n'est pas la formule 
(\ref{eqDefAdjficirc}) qui peut servir en pratique, mais plut\^{o}t la formule 
(\ref{eqDefAdjficirc2}).

\section{Modules \ptfs} 
\label{subsecPTF}
Cette section r\'esume un certain nombre de r\'esultats plus ou moins classiques. 
On trouve la plupart d'entre eux tr\`es bien expos\'es dans \cite{Nor}.
Pour des preuves  enti\`erement constructives ont peut consulter~\cite{LQ}. 
Rappelons qu'un module est dit \emph{\ptfz} s'il est isomorphe \`a un facteur 
direct dans un \Amo libre de dimension finie.

\subsection{Id\'eaux de Fitting  et \alis \lnls} 

Si $G\in\A^{m\times n}$, le module $\Coker(G)$ est dit \emph{\pfz}. 
Plus \gnlt on dit que la matrice $G$ est  \emph{une pr\'esentation d'un module 
$M$} si on a des \gtrs $g_1,\ldots ,g_m$ de $M$ et si l'application
$\A^m\rightarrow M$ qui envoie la base canonique sur les $g_i$ identifie
$M$ et $\Coker\,G$, \cade si les colonnes de $G$ engendrent le module des 
relations entre les $g_i$.

\begin{definition} \label{defIdeFit} Si $G$ est une matrice de pr\'esentation d'un 
module $M$ donn\'e par $m$ \gtrs li\'es par $n$ relations, les {\em \idfs du 
module} $M$ sont les \ids 
$$\cF_k(M):= \cD_{m-k}(G)$$
o\`u $k$ est un entier arbitraire. Ces \ids ne d\'ependent que de $M$ et non de la 
pr\'esentation choisie pour $M$.
\end{definition}

\begin{definition} 
\label{defEcm} 
Des \elts $x_1,\ldots ,x_\ell$ de $\A$ sont dit \emph{\comz} s'ils engendrent $\A$ 
comme id\'eal, \cad si une \coli des $x_i$ est \'egale \`a $1$. On dit encore que 
le vecteur  $(x_1,\ldots ,x_\ell)$ est \emph{unimodulaire} et que le \pol 
$P(Z)=\sum_kx_kZ^{k-1}$ est \emph{primitif}.
\end{definition}

Un but essentiel de l'article pr\'esent est de r\'ealiser avec un petit nombre 
d'\ops \elrs les \'equivalences annonc\'ees dans le \tho suivant.
Cela peut \^{e}tre compris comme donnant une solution uniforme et en temps 
raisonnable pour les \slis suffisamment \gui{bien conditionn\'es},
\`a l'image de ce que fait l'analyse num\'erique matricielle au moyen des 
d\'ecompositions en valeurs singuli\`eres (SVD) et des \iMPsz. 

Ce \tho est pour l'essentiel dans \cite{Nor} 
(lemme 1 page 8, exercice 7 page 49 et \tho 18 page 122)
et dans \cite{Bha} (voir aussi \cite{LQ}). La preuve dans \cite{Bha} est 
compl\`etement explicite,
contrairement \`a celle dans \cite{Nor}. 
De mani\`ere un peu surprenante, \cite{Bha} le fait remonter \`a \ldots 1994!

\begin{theorem} 
\label{theoremIFD}
Soit une \ali $\varphi:\A^n\rightarrow \A^m$  et $A$ sa matrice sur les bases 
canoniques. \Propeq 
\begin{enumerate}
\item \label{a} $\Im\,\varphi$  est   facteur direct dans  $\A^m$.
\item \label{b}  $\Coker\,\varphi$  est  un \mptfz.
\item \label{c}  $\Im\,\varphi$  est   facteur direct dans  $\A^m$,  
$\Ker\,\varphi$  est   facteur direct dans  $\A^n$ et si $H$ est un 
suppl\'ementaire de $\Ker\,\varphi$, $\varphi$ r\'ealise un \iso de $H$ sur 
$\Im\,\varphi$.
\item \label{d} Il existe  $\varphi\bul :\A^m\rightarrow \A^n$   telle que 
$\A^n=\Ker\,\varphi\oplus\Im\,\varphi\bul $ et  $\A^m=\Ker\,\varphi\bul 
\oplus\Im\,\varphi$.
\item \label{e} Il existe  $\psi:\A^m\rightarrow \A^n$   v\'erifiant
$\varphi \circ \psi \circ\varphi =\varphi$.
\item \label{f} Il existe  $\psi:\A^m\rightarrow \A^n$   v\'erifiant
$\varphi \circ \psi\circ\varphi =\varphi$ et $\psi \circ \varphi \circ \psi =\psi 
$.
\item \label{g} Chaque \idd $\cD_k(\varphi)$ est \idmz.
\item \label{h} Chaque \idd $\cD_k(\varphi)$ est engendr\'e par un \idm $e_k$. 
Soit alors $r_k=e_k-e_{k+1}$. Les $r_k$ forment un syst\`eme fondamental d'\idms 
\ortsz. 
Pour tout mineur $\mu$ d'ordre $k$ de $A$,
sur le localis\'e $\A[1/(r_k\,\mu)]$  l'\ali $\varphi$ devient \nl de rang $k$.
\item \label{i} L'\ali $\varphi$ devient \nl apr\`es localisation en des \elts 
$x_i$ \comz.
\item \label{j} L'\ali $\varphi$ devient \nl apr\`es localisation en
n'importe quel \id maximal.
\end{enumerate}
\end{theorem}

En particulier un \mpf est \pro \ssi ses \idfs sont \idmsz.
Le point \ref{j} est \`a part, car il n'implique les autres qu'avec l'aide de 
l'axiome du choix. Les autres \'equivalences sont \covsz.

Les \'equivalences \ref{a}$\;\Leftrightarrow\;$\ref{b}
$\;\Leftrightarrow\;$\ref{c}$\;\Leftrightarrow\;$
\ref{d} sont naturelles.

Pour passer de \ref{d} \`a \ref{e} on constate que les restrictions
de $\varphi$ et $\varphi\bul $ aux sous-modules $\Im\,\varphi$ et  
$\Im\,\varphi\bul $ sont des \isosz. Nous donnerons un calcul \gui{rapide} de 
$\psi$ \`a partir de $\varphi$ et $\varphi\bul$ dans la section~\ref{subsecIVGN}. 

Pour passer de \ref{e} \`a \ref{f} on remarque que si $\psi_1$ v\'erifie \ref{e}
alors  $\psi=\psi_1\circ \varphi \circ \psi_1$ v\'erifie~\ref{f}.

Dans les conditions du \ref{f}, $\varphi \circ \psi$ est la \prn sur
$\Im\,\varphi$ \paralm \`a $\Ker\,\psi$ et $\psi \circ\varphi$ est la \prn sur 
$\Im\,\psi$ \paralm \`a $\Ker\,\varphi$.

Le point \ref{g} implique le point \ref{h} de mani\`ere imm\'ediate en tenant 
compte du lemme \ref{lemLib}. Le point \ref{h} implique trivialement le point 
\ref{i} et celui-ci  implique trivialement le point~\ref{j}. 

Pour montrer que \ref{i} implique \ref{f} on consid\`ere
l'\egt $ABA=A$ (o\`u $A$ et $B$ sont des matrices pour
$\varphi$ et $\psi$) comme une \'equation o\`u l'inconnue est $B$: elle est facile 
\`a r\'esoudre dans le cas o\`u $A$ d\'efinit une \ali \nlz, donc elle est 
r\'esolue localement. Il reste \`a recoller les solutions en utilisant la \coli 
des $x_i$ \'egale~\`a~1. 

On peut aussi passer assez directement de  
\ref{h} \`a \ref{f} gr\^{a}ce \`a 
l'\idt de Cramer (\ref{eqIGCram}).   Si on a 
$$\sum_{\alpha\in \cP_{k,m},\beta\in \cP_{k,n}} 
c_{\alpha,\beta}\,\mu_{\alpha,\beta}=e_k\,,$$ 
alors $\sum_{\alpha\in \cP_{k,m},\beta\in \cP_{k,n}} 
c_{\alpha,\beta}\,\mu_{\alpha,\beta}\,r_k=r_k$
et puisque la matrice est de rang $\leq k$ sur $\A[1/r_k]$ l'\idt
(\ref{eqIGCram}) fonctionne. Il suffit alors de poser:
\begin{equation} \label{eqIGCra2}
B \;=\;
\sum_k\,\left(\sum\nolimits_{\alpha\in \cP_{k,m},\beta\in \cP_{k,n}} 
r_k\,c_{\alpha,\beta}\,
\Adj_{\alpha,\beta}(A) \right)\,.
\end{equation}

Les matrices qui v\'erifient les propri\'et\'es du \tho
\ref{theoremIFD} sont celles qui d\'efinissent les \gui{meil\-leurs} \slisz:
ceux pour lesquels on peut exprimer une solution du \sli comme
une fonction lin\'eaire du second membre. Ce sont aussi les syst\`emes pour 
lesquels on a une bonne description de l'image et du noyau, aussi bien du point de 
vue direct que du point de vue dual.
Ces matrices, d\'ej\`a intensivement \'etudi\'ees dans \cite{Nor} sont dites 
\gui{regular} dans \cite{Bha} et \cite{Pra}. 
Cette terminologie remonte \`a Von Neuman, qui a \'etudi\'e les anneaux (non 
n\'ecessairement commutatifs) dans lesquels tous les \elts poss\`edent un 
\gui{\ingz}: il a utilis\'e pour cela le terme d'\elt r\'egulier.
L'ennui est que pour une matrice carr\'ee, \gui{r\'eguli\`ere} signifie
inversible dans la terminologie courante. 

Nous proposons la terminologie suivante, que nous \'etendrons  au cas des \mptfsz.
\begin{definition} 
\label{defLocNl} 
Une \ali $\varphi$ entre modules libres de dimensions finies qui v\'erifie les 
conditions \'equivalentes du \tho  \ref{theoremIFD} sera appel\'ee une \ali 
\emph{\lnlz}. Si $\cD_k(\varphi)=\gen{1}$  et 
 $\cD_{k+1}(\varphi)=\gen{0}$  on dira qu'il s'agit d'une \emph{\ali (\lnlz) de 
rang $k$}. Si on ne pr\'ecise pas le rang mais qu'il existe, on dit 
qu'il s'agit d'une \emph{\ali de rang constant}. Une \emph{matrice \lnl (resp. de 
rang constant)} est la matrice d'une \ali \lnl (resp. de rang constant).
\end{definition}

 Quand il existe, le rang d'une \ali \lnl est bien d\'efini d\`es que l'anneau 
n'est pas trivial.

Sur un anneau sans autre \idm que 1 et 0 toute matrice \lnl est de rang constant.
\subsection{Rang d'un \mptf} 

\begin{definition} 
\label{notaDk} Soit un \mptf $E$ engendr\'e par  $n$ \eltsz, et $\psi$ 
un \endo de $E$.
\begin{enumerate}
\item Le \deter de $\psi$, si $E\oplus N\simeq \A^n,$ est d\'efini par $\det\,\psi 
:=\det(\psi\oplus \Id_N)$.
\item Nous notons
 $$ \rP\psi(X)=\det(\Id_E+X\psi)=
1+\rd 1(\psi)X+\cdots+ \rd n(\psi)X^n
$$
En particulier $\rd 1(\psi)=\Tr(\psi)$ est appel\'e la trace de $\psi$. On pose 
aussi  $\rd0(\psi)=1$ et
pour $p>n$, $\rd{p}(\psi)=0$. 
\item Le \polcar de $\psi$ sera not\'e
$\rC\psi(X)=\det(X\Id_E-\psi)$. 
\item On note 
$\Gamma_{\psi}(X)$ 
le \pol d\'efini par 
$\;\rC{-\psi}(-X)=-X\Gamma_{\psi}(X)+\det(\psi)$, et   $\Gamma_{\psi}(\psi)$ 
s'appelle \emph{l'\endo cotranspos\'e de $\psi$}. Nous le notons $\wt{\psi}$ ou 
$\Adj\,\psi$.
\item Le  \deter de la multiplication par $X$ sur le module $E$, not\'e  
$\rR{E}(X),$
est appel\'e le \gui{\pol multiplicatif du module}.
\end{enumerate}
\end{definition}

Dans la \dfn pr\'ec\'edente il est sous-entendu que le d\'eterminant est \gui{bien 
d\'efini}: il ne d\'epend ni de l'entier $n$ ni de la d\'ecomposition $E\oplus 
N\simeq \A^n$.

L'\endo $\psi$  est inversible \ssi $\det\,\psi$ est inversible. 
Le \tho de Cayley-Hamilton est valable pour les \mptfsz.

Si un anneau poss\`ede des \idms un \mptf n'a pas forc\'ement un rang bien 
d\'efini. C'est le \gui{\pol multiplicatif} qui remplace le rang.

\begin{theorem} 
\label{propPTFDec} 
Soit $E$ un \mptf isomorphe \`a l'image d'une matrice de projection
$P\in\A^{m\times m}$ (de sorte que $\I_m-P$ est une matrice de pr\'esentation de 
$E$). 
\begin{enumerate}
\item Le \pol $\rR{E}(X)=\sum_ir_iX^i$ est \'egal \`a $\det(\I_m+(X-1)P)$. C'est 
un
\pol \emph{multiplicatif}: il
v\'erifie $\rR{E}(XY)=\rR{E}(X)R_E(Y)$  et $\rR{E}(1)=1$. Cela signifie que les
$r_i$  forment un syst\`eme fondamental d'\idms orthogonaux (sfio).
On a $r_0=\det(\I_m-P)$ et l'id\'eal $\gen{r_0}$ est l'annulateur de~$E$.
\item  Le module est dit de rang $k$ si $\rR{E}(X)=X^k$. Il est dit de rang $\leq 
k$ si $r_{k+1}=\ldots =r_m=0$ \cad encore si tous les mineurs d'ordre $k+1$ de la 
matrice $P$ sont nuls. S'il a un rang $k$ le module $E$ est dit de rang constant.
Quand il existe, le rang d'un \mptf est bien d\'efini d\`es que l'anneau n'est pas 
trivial.
\item Le localis\'e $\A_{r_k}=\A[1/r_k]$ est isomorphe \`a  
$\aqo{\A}{1-r_k}$. Le  localis\'e $E_{r_k}$ est isomorphe 
au sous-module $r_kE.$  Il est de rang $k$ en tant que  
$\A_{r_k}$-module. Le module
$E$ est somme directe des \gui{composantes} $r_kE$ ($k>0$).
\item Les \idfs $\cF_k(E)$ sont li\'es aux \idms $r_k$ d\'efinis via le \pol 
multiplicatif $\rR{E}$ par la relation:
 $\cF_k(E)=\gen{\sum_{\ell\leq k}r_\ell}$.
\item Les $m\choose k$ mineurs diagonaux d'ordre $k$ de $P$ sont les 
$\mu_{\alpha,\alpha}=\mu_{\alpha,\alpha}(P)$ pour $\alpha\in\cP_{k,m}$ et on a 
$\rd k(P)=\sum_{\alpha \in\cP_{k,m}}\mu_{\alpha ,\alpha }$.
Alors $r_k\,\rd k(P)=r_k$ 
et, pour chaque $\alpha\in\cP_{k,m} $, le module $E$ devient libre de rang $k$ 
lorsqu'on localise en $r_k\,\mu_{\alpha ,\alpha }$.
\end{enumerate}
\end{theorem}

 Des r\'esultats utiles qui am\'eliorent l\'eg\`erement la proposition \ref{prop 
quot non iso} sont donn\'es dans  la proposition suivante et son corollaire (pour 
une \prco voir \cite{LQ}).
\begin{proposition} 
\label{prop epi rang constant}
Soit $\varphi : E\rightarrow F$ une \ali  {\em  surjective} entre \mptfs  de 
m\^{e}me \pol multiplicatif $\rR{E}=\rR{F}$, alors  $\varphi$ est un \isoz. Ceci 
s'applique en particulier s'ils ont m\^{e}me rang constant.
\end{proposition}

\begin{corollary} 
\label{corEpiRC}  Soit $F$ un \mptfz.
Si $F_1\oplus G_1 = F =F_2\oplus G_2$ avec $F_1\subset F_2$ on~a:

\sni\centerline{
 $\rR{F_1}=\rR{F_2}\;\Longleftrightarrow \;\rR{G_1}=\rR{G_2}\;\Longleftrightarrow 
\;F_1=F_2$} 
\end{corollary}
\subsection{Quand les \mrcs sont libres}\label{subsecMrcLibre} 
Pour un anneau $\A$ il revient au m\^{e}me de dire que toutes
les matrices \lnls de rang constant sont \nlsz, ou que tous les
modules \pros de rang constant sont libres.

Signalons quelques cas importants o\`u ceci se produit.
\begin{itemize}
\item $\A$ est un anneau local. 
\item $\A$ est z\'ero-dimensionnel 
(i.e., $\forall x\in\A,\,\exists y\in\A,\,\exists n\in\NN\;\;x^{n+1}=yx^n$).   
\item Le quotient de $\A$ par son radical de Jacobson $\Rad(\A)$ est 
z\'ero-dimensionnel (rappelons que si $\cU_\A$ d\'esigne le groupe des unit\'es de 
$\A$, $\Rad(\A)=\{\,x\in\A\,|\,1+x\A\subset\cU_\A\,\}$).
\item $\A$ est \gui{fortement U-irr\'eductible}, i.e., 
pour tout \pol  $P=\sum_{i=0}^na_ix^i\in\A[X]$ primitif (i.e., tel que 
$\gen{a_0,\ldots ,a_n}=\gen{1}$), il existe $x\in\A$ tel que $\gen{P(x)}=\gen{1}$.
\item $\A=\gB[X_1,\ldots ,X_n]$ o\`u $\gB$ est un anneau de Bezout, i.e.,
tout \itf de $\gB$ est principal.
\end{itemize}

Les trois premiers cas sont trait\'es \cot dans \cite{LQ}.

Il semble qu'on ne connaisse pas pour le moment de preuve constructive pour le 
dernier cas, qui est une extension remarquable du \tho de Quillen-Suslin, due \`a 
Lequain et Simis \cite{LS}. 
Une telle preuve fournirait un \algo pour transformer une matrice $A$ \lnl de rang 
constant $r$ en une matrice  
$\I_{r,m,n}=Q_1AQ_2$ avec  $Q_1$ et $Q_2$  inversibles.

Le quatri\`eme cas est assez facile. 
Rappelons comment cela fonctionne.
On part d'une matrice \lnl $A=(a_{ij})\in\A^{m{\times}n}$ de rang 
$r\geq 1$. On va la diagonaliser par des changements de base de la source et du 
but. 
Imaginons que nous multiplions chaque ligne \num$k$ par $X^{k-1}$ puis chaque 
colonne \num$\ell$ par $X^{(\ell-1) m}$, et consid\'erons le \pol 
$P(X)=\sum_{k\ell}a_{k\ell}X^{(\ell-1) m+k-1}$. Par hypoth\`ese ce \pol est 
primitif. Soit  $x\in\A$ tel que $P(x)$ est inversible.
Dans la matrice $A$ on ajoute \`a la premi\`ere ligne $L_1$ les lignes $x^{k-
1}L_k$ ($k>1$). Puis dans la matrice obtenue, 
on ajoute \`a la premi\`ere colonne $C_1$ les colonnes  
$x^{(\ell-1)m}\,C_\ell$ ($\ell>1$). Alors en position $(1,1)$ on trouve 
$P(x)$ qui peut servir de pivot de Gauss. On termine par induction.

Notez que l'\algo esquiss\'e ci-dessus utilise un nombre raisonnable
d'\ops \elrs si le caract\`ere fortement U-irr\'eductible de l'anneau
est rendu explicite au moyen d'un nombre raisonnable d'\ops
\elrs (nous ne cherchons pas \`a formaliser la chose).

Pour tout anneau $\A$ il y a une extension fid\`element plate et 
fortement U-irr\'eductible de
$\A$ qui est le localis\'e \gui{de Nagata} $\A(X)=S^{-1}\A[X]$ o\`u 
$S\subset\A[X]$ est le mono\"{\i}de form\'e par les \pols primitifs.
En effet, si $P(T)=\sum_iQ_i(X)T^i$ est tel que $\sum_iB_i(X)Q_i(X)$ 
soit un \pol primitif alors pour $k>\sup_i(\deg_X(Q_i))$, $P(X^k)$ 
est lui-m\^{e}me un \pol primitif.

Il n'est pas \'etonnant que cet anneau  $\A(X)$ joue un role crucial dans la suite 
pour nos calculs uniformes (en temps raisonnable).

\junk{
\subsection{Calculs avec un petit nombre d'\ops \elrs} 

Pour appliquer le \tho \ref{propPTFDec}, il faut conna\^{\i}tre
la matrice de projection $P$.
En pratique si on est parti d'une matrice $A$ pour l'\ali $\varphi$ du \tho 
\ref{theoremIFD} et si on est capable de calculer une matrice $B$ pour un $\psi$ 
v\'erifiant le point \ref{e} de ce \thoz, on aura $P=AB$.

Le \pb de calculer $\psi$ \`a partir de $\varphi$ est de toute mani\`ere important 
\`a r\'esoudre en pratique, car il permet de donner une solution uniforme
au \sli $\varphi(x)=b$.
Trouver $\psi$ v\'erifiant le point \ref{e} revient de nouveau \`a r\'esoudre un 
\sli (dont les inconnues sont les entr\'ees de la matrice $B$ de $\psi$).
La m\'ethode propos\'ee implicitement ci-dessus 
(dans le commentaire qui suit le \tho \ref{theoremIFD}) pour faire ce travail,
en utilisant les \idts de Cramer pour d\'eduire \ref{e} de \ref{h} est 
constructive mais pas praticable. 
On propose en effet dans un premier temps de calculer les \idms $e_k$ tels que 
$\cD_k(A)=\gen{e_k}$, d'en d\'eduire les $r_k$ puis de faire une \coli d'\idts de 
Cramer.
Ce calcul n'est gu\`ere praticable \`a cause du trop grand nombre de \gtrs des 
\ids  $\cD_k(A)$.

}
\subsection{Applications lin\'eaires \lnls entre \mptfsz}

\begin{definition} 
\label{defIng2} 
Soient $E$ et $F$ deux \Amos \ptfsz,  et une \ali $\varphi:E\rightarrow F$. 
\begin{enumerate}
\item Les {\em  \idds de l'\ali $\varphi$}  sont les id\'eaux 

\ss\centerline{$\cD_k(\varphi ) \; 
:=  \;$  l'id\'eal   engendr\'e  par  les   $\det(\lambda \circ\varphi 
\circ\theta)$  
}

\sni o\`u $k$ est un entier arbitraire, $\lambda :F\rightarrow \A^k$ et
$\theta:\A^k\rightarrow E$ sont arbitraires. 
\item L'\ali $\varphi$ est dite de rang $\leq k$ si $\cD_{k+1}(\varphi)=0$.
\item L'\ali $\varphi$ est dite \lnl si $\Im\,\varphi$ est facteur direct dans 
$F$. Elle est dite  \lnl de rang $k$ si en outre le module \pro $\Im\,\varphi$ est  
de rang $k$. 
\end{enumerate}
\end{definition}

Le calcul de $\cD_k(\varphi )$ se fait comme suit. Supposons que $E\oplus E_1 
\simeq \A^n$,  $F\oplus F_1 \simeq \A^m$, que
$\Pi_E: \A^n\rightarrow E$ est la projection sur $E$ \paralm \`a $E_1$, $\iota_F: 
F\rightarrow \A^m$ est l'injection naturelle et que $\varphi'=\iota_F\circ\varphi 
\circ\Pi_E$. Alors
$\cD_k(\varphi )=\cD_k(\varphi')$.
Si $P\in\A^{n{\times}n}$ est la matrice de  la projection sur $E$ \paralm \`a 
$E_1$  (\cad de $\pi_E=\iota_E\circ\Pi_E$) et $Q\in\A^{m{\times}m}$ est la matrice 
de  la projection sur $F$ \paralm \`a $F_1$ alors, la matrice $A$ de $\varphi'$ 
v\'erifie $QAP=A$ (et cette
\egt caract\'erise les matrices du type $\varphi'$).
On dira que la matrice $A$ \emph{repr\'esente} $\varphi$ (via les \isos
$E\oplus E_1 \simeq \A^n$ et  $F\oplus F_1 \simeq \A^m$).

De mani\`ere \gnle tout calcul sur les \mptfs se ram\`ene \`a un calcul sur des 
matrices.

\smallskip Une \ali $\varphi$ entre \mptfs est \lnl de rang $k$ \ssi  
$\cD_k(\varphi)=\gen{1}$ et  $\cD_{k+1}(\varphi)=0$. Elle est \lnl \ssi tous ses 
\idds sont engendr\'es par des \idmsz. Plus \gnlt on peut recopier le \tho 
\ref{theoremIFD}. 

\begin{theorem} 
\label{theoremIFD2}
Les propri\'et\'es suivantes pour une \ali $\varphi:E\rightarrow F$   entre \mptfs 
sont \'equivalentes.
\begin{enumerate}
\item \label{a'} $\Im\,\varphi$  est   facteur direct dans  $F$.
\item \label{b'}  $\Coker\,\varphi$  est   un \mptfz.
\item \label{c'}  $\Im\,\varphi $  est   facteur direct dans  $F$,  $\Ker\,\varphi 
$  est   facteur direct dans  $E$ et si $H$ est un suppl\'ementaire de 
$\Ker\,\varphi $, $\varphi$ r\'ealise un \iso de $H$ sur $\Im\,\varphi $.
\item \label{d'} Il existe  $\varphi\bul :F\rightarrow E$   telle que 
$E=\Ker\,\varphi\oplus\Im\,\varphi\bul $ et  $F=\Ker\,\varphi\bul 
\oplus\Im\,\varphi$.
\item \label{e'} Il existe  $\psi:F\rightarrow E$   v\'erifiant
$\varphi \circ \psi \circ\varphi =\varphi$.
\item \label{f'} Il existe  $\psi:F\rightarrow E$   v\'erifiant
$\varphi \circ \psi\circ\varphi =\varphi$ et $\psi \circ \varphi \circ \psi =\psi 
$.
\item \label{g'} Chaque \idd $\cD_k(\varphi)$ est \idmz.
\item \label{h'} Chaque \idd $\cD_k(\varphi)$ est engendr\'e par un \idm $e_k$. 
Soit alors $r_k=e_k-e_{k+1}$. Les $r_k$ forment un syst\`eme fondamental d'\idms 
\ortsz. 
Pour tout mineur $\mu$ d'ordre $k$ d'une matrice $A$ qui repr\'esnte $\varphi$, 
sur le localis\'e $\A[1/(r_k\,\mu)]$  l'\ali $\varphi$ devient \nl de rang $k$.
\item \label{i'} L'\ali $\varphi$ devient \nl apr\`es localisation en des \elts 
$x_i$ comaximaux.
\item \label{j'} L'\ali $\varphi$ devient \nl apr\`es localisation en
n'importe quel \id maximal.
\end{enumerate}
\end{theorem}

\section{Applications \lins \cros et \ings pour les \mptfs} 
\label{subsecIVGN}

Dans toute la section \ref{subsecIVGN}, $E$ et $F$ sont des \mptfsz.
\subsection{Un crit\`ere pour les \alis crois\'ees 
avec elles-m\^{e}mes}

Voici une \gnn d'un r\'esultat usuel pour les \evcs de dimension finie. Il s'agit 
ici d'une cons\'equence importante du \tho de Cayley-Hamilton.

\begin{theorem} 
\label{thCHam} 
Soit $\varphi :E\rightarrow E$ un \endoz. 
Notons $d_j=\rd j(\varphi)$.
\Propeq
\begin{enumerate}
\item $\varphi$ est de rang $\leq k$  et $d_k$ est inversible.
\item $\varphi$ est crois\'ee avec elle-m\^{e}me et de rang $k.$
\end{enumerate}
Lorsque ces conditions sont v\'erifi\'ees,    la projection
$\pi:E\rightarrow E$ sur $\Im\,\varphi$ \paralm \`a $\Ker\,\varphi$ v\'erifie:
\begin{equation} \label{eqthCHam}
 d_k\,\pi = d_{k-1}\,\varphi -d_{k-2}\,\varphi^2+
\cdots +(-1)^{k-1}\varphi^{k}\,.
\end{equation}
En outre l'\ing $\psi=\Ig(\varphi,\varphi)$ v\'erifie
\begin{equation} \label{eqthCHam2}
 d_k\,\psi  = d_{k-1}\,\pi -d_{k-2}\,\varphi+d_{k-3}\,\varphi^2+
\cdots +(-1)^{k-1}\varphi^{k-1}\,.
\end{equation}
\end{theorem}
\begin{proof}{Preuve}
Le point d\'elicat est: 1 implique 2.\\
On a $d_ k\in\cD_k(\varphi)$ donc $\cD_k(\varphi)=\gen{1}$, et  
$\cD_{k+1}(\varphi)=0$ par hypoth\`ese. Donc  $\varphi$ est \lnl de rang $k$.
Soit $K$ un suppl\'ementaire de $\Im\,\varphi$ dans $E$. Sur cette somme
directe $\varphi$ est \gui{triangulaire} avec une \gui{matrice} du type:
$$ \cmatrix{\varphi_0&\varphi'\cr
0_{\Im\,\varphi,K}&0_{K,K}}
$$
o\`u $\varphi_0:\Im\,\varphi\rightarrow\Im\,\varphi$  est la restriction de 
$\varphi$. Donc $\det(\Id_E+X\varphi)=\det(\Id_{\Im\,\varphi}+X\varphi_0)$.
On obtient $\det\,\varphi_0=d_ k$ et donc $\varphi_0$ est inversible. Ceci 
implique tout d'abord $\Im\,\varphi \cap\Ker\,\varphi =0$. Ensuite
tout $x\in E$ s'\'ecrit $x_1+x_2$ o\`u $x_1=\varphi_0^{-
1}(\varphi(x))\in\Im\,\varphi$ et $x_2=x-x_1\in\Ker\,\varphi$.
Donc  $E=\Im\,\varphi \oplus \Ker\,\varphi$. On peut donc remplacer $K$ par 
$\Ker\,\varphi$ et la \gui{matrice} ci-dessus devient \gui{diagonale} 
($\varphi'=0_{\Ker\,\varphi ,\Im\,\varphi}$). 
Le \tho de Cayley-Hamilton appliqu\'e \`a $\varphi_0$ donne 
$$ d_ k\,\Id_{\Im\,\varphi } = \varphi_0\,\left(d_{k-1} 
-d_{k-2}\,\varphi_0+ d_{k-3}\,\varphi_0^2+
\cdots +(-1)^{k-1}\varphi_0^{k-1}\,\right)$$
ce qui implique facilement les \egts voulues.
\end{proof}

\subsection{Applications \lins \cros entre \mptfsz}   
\label{LCPTF}

Dans toute la suite de la section \ref{subsecIVGN} on consid\`ere deux \alis 
$\varphi:E\rightarrow F$ et $\varphi\bul:F\rightarrow E$.

Supposons tout d'abord $\varphi$ et $\varphi\bul$ \crosz. Nous reprenons les 
notations de la section \ref{subsecIGALC}.

Soit $\varphi_1:\Im\,\varphi\rightarrow \Im\,\varphi$ 
l'\auto \lin d\'efini par
$\varphi_1=\varphi_0\,\varphi_0\bul$. C'est la restriction 
de $\varphi\,\varphi\bul$ \`a $\Im\,\varphi$. 
D\'efinissons de m\^{e}me $\varphi_1\bul:\Im\,\varphi\bul\rightarrow 
\Im\,\varphi\bul$ par $\varphi_1\bul=\varphi_0\bul\,\varphi_0$.
Si le rang de $\varphi$ est $\leq k$ on obtient alors:
$$
\rP{\varphi_1\bul}(Z)
\;=\;\rP{\varphi\bul\varphi}(Z)
\;=\;\rP{\varphi\varphi\bul}(Z)
\;=\;\rP{\varphi_1}(Z)\;=\;1+ a_1 Z+ \cdots+ a_kZ^k
$$
o\`u $a_j=\rd j(\varphi\bul \varphi)$. Si  $\varphi$  est de rang constant $k$ 
alors il en va de m\^{e}me pour
$\varphi_0$, $\varphi_1$, $\varphi \bul$, $\varphi \,\varphi \bul$ etc\ldots.
  De sorte que  $a_k$ est un \elt inversible de $\A$.

On a la r\'eciproque suivante importante, qui est un analogue du \tho
\ref{thCHam}.

\begin{theorem} 
\label{thCroiRgCst} 
Notons $a_j=\rd j(\varphi\varphi\bul)$. \Propeq
\begin{enumerate}
\item  $\varphi$ et $\varphi\bul$ sont \cros de rang $k$.
\item  $\varphi$ et $\varphi\bul$ sont de rang $\leq k$  et $a_k$ est inversible.
\end{enumerate}
\end{theorem}
\begin{proof}{Preuve}
Il faut montrer la r\'eciproque. On a $a_k\in\cD_k(\varphi \varphi\bul)\subset 
\cD_k(\varphi)\, \cD_k(\varphi\bul)$. Puisque $a_k$ est inversible 
$\cD_k(\varphi)= \cD_k(\varphi\bul)=\gen{1}$ donc $\varphi$ et $\varphi\bul$ sont
\lnls de rang $k$.\\
Le \tho \ref{thCHam} montre en outre que $\varphi \varphi\bul:F\rightarrow F$ 
est \cro avec elle m\^{e}me, de rang $k$. M\^{e}me chose
pour $\varphi\bul \varphi:E\rightarrow E$. 

\noi On a donc la situation suivante: $F_1=\Im\,\varphi\, \varphi\bul\subset 
F_2=\Im\,\varphi\subset F$  sont
2 modules de rang $k$ en facteur direct dans $F$. On peut donc appliquer 
le corollaire \ref{corEpiRC}:  $F_1=F_2$. 
Ainsi $\Im\,\varphi \,\varphi\bul =\Im\,\varphi$ et sym\'etriquement
 $\Im\,\varphi\bul \,\varphi =\Im\,\varphi\bul$.

\noi De la m\^{e}me fa\c{c}on on a  $K_1=\Ker\, \varphi\bul\subset 
K_2=\Ker\,\varphi\,\varphi\bul\subset E$. Ils sont tous deux en facteur direct 
avec un suppl\'ementaire de rang $k$. On peut donc appliquer 
le corollaire \ref{corEpiRC}:  $K_1=K_2$.  

\noi Finalement on obtient $E=\Im\,\varphi\bul\oplus\Ker\,\varphi$ et 
$F=\Im\,\varphi\oplus\Ker\,\varphi\bul$.
\end{proof}

\subsection{Calcul th\'eorique d'un \ingz: le cas du rang constant}      
\label{IVGNRC}

En utilisant le \tho de Cayley-Hamilton on d\'emontre comme pour le \tho 
\ref{thCHam} le r\'esultat suivant. 
\begin{theorem} 
\label{thIgRangr} 
{\rm  (\prns sur l'image et sur le 
noyau et \ing en rang constant $k$)}\\
 Si  $\varphi$ et $\varphi\bul$ sont \cros de rang $k$, avec $a_j=\rd j(\varphi 
\varphi\bul)$, on a:
\begin{enumerate}
\item L'\ing de $\varphi$ via $\varphi\bul$ 
est donn\'e par 
\begin{equation} \label{eqIg2}
\psi =\Ig(\varphi,\varphi\bul)=\varphi^{\dag_{\varphi\bul}}=a_k^{-1}\left(a_{k-
1}\varphi \bul-a_{k-2}\varphi \bul\varphi \varphi \bul  + \cdots+
(-1)^{k-1}(\varphi \bul\varphi )^{k-1}\varphi \bul\right) .
\end{equation}
\item La  \prn  
sur le sous-espace
$I=\Im\,\varphi  \subseteq F$ \paralm \`a 
$\Ker\,\varphi \bul$
est \'egale \`a  $\pi_I=\varphi \,\psi.$
\item La  \prn  sur le sous-espace
$I\bul=\Im\,\varphi \bul \subseteq E$ 
 \paralm \`a $\Ker\,\varphi $
est \'egale \`a $\pi_{I\bul}=\psi \,\varphi.$
Et la \prn sur le noyau de $\varphi$ 
 \paralm \`a $\Im\,\varphi \bul$ est 
$\Id_{E}-\pi_{I\bul}$.
\end{enumerate}
\end{theorem}

On obtient aussi l'\'equivalence g\'en\'erale suivante: 
\begin{theorem} 
\label{corthIgRangr} 
Notons $a_j=\rd j(\varphi\varphi\bul)$.
\Propeq
\begin{enumerate}
\item  $\varphi$ et $\varphi\bul$ sont crois\'ees de rang $k$.
\item  $a_h=0$ pour $h>k$, $a_k$ est inversible et, en d\'efinissant $\theta$ par

\sni\centerline{$
\theta =a_{k-1}\, \varphi \bul
-a_{k-2}\,\varphi \bul\varphi \varphi \bul+
\cdots +(-1)^{k-1}\varphi \bul(\varphi \varphi \bul)^{k-1}$,}

\sni on a les deux \egts $\varphi\,\theta\,\varphi=a_k\,\varphi$ et 
$\varphi\bul\,\varphi \,\theta =a_k\,\varphi\bul$.
\end{enumerate}
\end{theorem}
\begin{proof}{Preuve}
Il reste \`a montrer que 2 implique 1. Posons

\sni\centerline{  $\psi =a_k^{-1}\theta\quad \mathrm{et}\quad \psi\bul=a_k^{-
1}\left(a_{k-1}\varphi \,-\,a_{k-2}\,\varphi \varphi\bul\varphi   \,+ \,\cdots+
(-1)^{k-1}(\varphi \varphi\bul )^{k-1}\varphi \right)$.~~~~~}

\sni Un calcul simple montre que les six \egts de la proposition \ref{thCrIvg2} 
sont satisfaites.
\end{proof}

On peut utiliser le test pr\'ec\'edent pour savoir si une \ali est crois\'ee avec 
elle-m\^{e}me et de rang $k$. Une l\'eg\`ere variante, sans doute plus efficace du 
point de vue du calcul est obtenue de fa\c{c}on analogue en s'appuyant sur le
\tho \ref{thCHam} et la proposition \ref{thCrIvg2bis}:

\begin{theorem} 
\label{corthIgRangr2} 
On suppose $F=E$. Notons $d_j=\rd j(\varphi)$. \Propeq
\begin{enumerate}
\item  $\varphi$ est crois\'ee avec elle-m\^{e}me, de rang $k$.
\item  $d_h=0$ pour $h>k$,  $d_ k$ est inversible et, en d\'efinissant $\pi$ par

\sni\centerline{$
\pi=d_{k-1}\,\varphi-d_{k-2}\,\varphi ^2+
\cdots +(-1)^{k-1}\varphi^k,\quad \quad ~$}

\sni on a les \egts $\pi\,\varphi=d_ k\,\varphi$ et $\pi^2=d_k\,\pi$.
\end{enumerate}
\end{theorem}

Le \tho \ref{corthIgRangr} conduit au r\'esultat de complexit\'e suivant,
lorsque les modules $E$ et $F$ sont donn\'es par des matrices de projection $P_E$ 
et $P_F$  dont ils sont les images, et
$\varphi$ et $\varphi\bul$ sont donn\'ees par des matrices $A$ et $A\bul$ 
v\'erifiant $P_FAP_E=A$ et $P_EA\bul P_F=A\bul$. 

\begin{theorem} 
\label{thDecisionO} 
Soient $E$ et $F$ deux \Amos \ptfs engendr\'es par $n$ \elts (ou moins), et deux 
\alis $\varphi:E\rightarrow F$ et $\varphi\bul:F\rightarrow E$. Alors on peut, 
avec  $\cO(n^4)$
\oparisz, un test \gui{$1\in\gen{y}\,?$} et $\cO(n^2)$  tests \gui{$x=0\,?$}, 
d\'ecider si $\varphi$ et $\varphi\bul$ sont \cros et de rang $k$, et en cas de 
r\'eponse positive calculer les \ings $\Ig(\varphi,\varphi\bul)$ et 
$\Ig(\varphi\bul,\varphi)$ en utilisant $\cO(n^4)$ \oparisz. 
\junk{En outre tout le calcul peut \^{e}tre fait en $\cO(\log^2n)$ \'etapes de 
calcul \paralz.}
\end{theorem}

\subsection{Cas o\`u le rang n'est pas constant}      
\label{subsecGenRangNonConst}

Les r\'esultats de la section \ref{IVGNRC} se g\'en\'eralisent en cassant l'anneau 
en des composantes convenables donn\'ees par un sfio.

L'id\'ee g\'en\'erale est la suivante: si $\rR{\Im\,\varphi}(X)=\sum_i r_i X^i$ et 
$\rP{\varphi \varphi\bul}(Z)=1+ a_1 Z+ \cdots + a_n Z^n=P(Z)$, chaque \pol 
$P_k=r_kP$ doit \^{e}tre de degr\'e $k$ avec pour \coe de $Z^k$ un \elt inversible 
dans $\A[1/r_k]\simeq r_k\A$ et ceci permet de calculer les $r_i$ lorsqu'on 
conna\^{\i}t $P$.  Plus pr\'ecis\'ement
\begin{itemize}
\item On peut retrouver les $r_i$ \`a partir des $a_i$. Par exemple
on doit avoir
$\gen{r_n}=\gen{a_n}=\gen{a_n^2}$: on teste si $\gen{a_n}=\gen{a_n^2}$, en cas de 
r\'eponse positive avec $a_n=b_na_n^2$, alors $r_n=a_nb_n$, 
puis on recommence avec $(1-r_n)P$ pour trouver
$r_{n-1}$ et ainsi de suite.  
\item On pose $\varphi_k=r_k\,\varphi$ et $\varphi_k\bul=r_k\,\varphi\bul$. Par le 
\tho \ref{thDecisionO} on peut tester si $\varphi_k$ et $\varphi\bul_k$ sont 
crois\'ees et de rang $k$ sur l'anneau $\A[1/r_k]$, et en cas de r\'eponse 
positive calculer l'\ing $\Ig(\varphi_k,\varphi_k\bul)$.
\item On termine en recollant tout ceci: $\Ig(\varphi,\varphi\bul)= 
\sum_{k=1}^nr_k\,\Ig(\varphi_k,\varphi_k\bul).$
\end{itemize}

Le calcul le plus long dans toute cette affaire est celui du \pol $P$ 
qui se fait en $\cO(n^4)$ \oparis dans $\A$. 

La proc\'edure enti\`ere est explicite si on dispose d'un test de divisibilit\'e 
dans $\A$,
c.-\`a-d. un test pour \gui{$x\in\gen{y}\,?$} qui donne un $z$ tel que $x=zy$ en 
cas de r\'eponse positive. En faisant \gui{$x-y\in\gen{0}\,?$}
on a aussi un test pour l'\egt dans $\A$.

\smallskip On r\'esume la situation dans le \tho suivant, en supposant que les 
modules $E$ et $F$ sont donn\'es par des matrices de projection $P_E$ et $P_F$  
dont ils sont les images, et que
$\varphi$ et $\varphi\bul$ sont donn\'ees par des matrices $A$ et $A\bul$ 
v\'erifiant $P_FAP_E=A$ et $P_EA\bul P_F=A\bul$.   

\begin{theorem} 
\label{thDecision1} 
Soient $E$ et $F$ deux \Amos \ptfs engendr\'es par $n$ \elts (ou moins), et deux 
\alis $\varphi:E\rightarrow F$ et $\varphi\bul:F\rightarrow E$. Alors on peut, 
avec un nombre
d'\oparis en $\cO(n^4)$, et un nombre de tests \gui{$x\in\gen{y}\,?$} en 
$\cO(n^3)$, d\'ecider si $\varphi$ et $\varphi\bul$ sont \crosz, et en cas de 
r\'eponse positive calculer les \ings $\Ig(\varphi,\varphi\bul)$ et 
$\Ig(\varphi\bul,\varphi)$ en $\cO(n^4)$ \oparisz. 
\end{theorem}
\section{Calcul pratique d'un \ing s'il en existe un.} 
\label{subsecCRIG}

Dans cette section nous g\'en\'eralisons au cas 
d'un anneau commutatif $\A$
le travail que nous avons fait dans \cite{DiGL} en vue de la r\'esolution uniforme 
des \slis sur un corps arbitraire, en nous appuyant sur un calcul uniforme du rang 
d'une matrice d\^{u} \`a Mulmuley~\cite{Mul}.
Il s'agit de la possibilit\'e de calculer efficacement un \ing d'une \ali
entre \Amos libres lorsqu'il en existe un. En termes plus abstraits:
lorsqu'on conna\^{\i}t une matrice de pr\'esentation pour un module $E$,
on est capable de tester si ce module est \ptf et en cas de r\'eponse positive, de 
fournir une matrice de \prn dont l'image est isomorphe \`a $E$ (l'\iso est 
explicite). Tout ceci avec des calculs assez efficaces, \cad ici \emph{en temps 
polynomial}.

\smallskip Consid\'erons une \ali $\varphi$ entre deux \Amos libres $E$ et $F$ de 
dimensions respectives $n$ et $m$.
Notre but est de donner un test pour savoir si $\varphi$ est \lnlz, et, en cas de 
r\'eponse positive, de calculer en temps polynomial un \ing de la matrice.

Bien que nous traitions uniquement le cas des modules libres, il ne serait pas
difficile de g\'en\'eraliser au cas o\`u $E$ et $F$ sont des \mptfsz.

\smallskip Nous nous limitons au point de vue purement 
matriciel, (c'est le point de vue o\`u des bases ont \'et\'e 
fix\'ees dans $E$ et $F$). Nous introduisons une \idtr $t$.
Nous consid\'erons une forme quadratique $\Phi_{t,n}$ sur
$E'=\A(t)^n$ et une forme quadratique $\Phi_{t,m}$ sur
$F'=\A(t)^m$:  
$$
\begin{array}{rcl} 
\Phi_{t,n}(\xi_1,\ldots ,\xi_n)& =  &{\xi_1}^2+t\, {\xi_2}^2 
+ \cdots + t^{n-1}\,{\xi_n}^2    \\[1mm] 
\Phi_{t,m}(\zeta_1,\ldots ,\zeta_m)& =  & {\zeta_1}^2 + 
t\,{\zeta_2}^2 + \cdots + t^{m-1}\,{\zeta_m}^2    
\end{array}
$$

Nous notons les \gui{produits scalaires}  correspondants par
$\gen{\cdot,\cdot}_{E'}^{t}$ et 
$\gen{\cdot,\cdot}_{F'}^t$. Nous notons $Q_n$
et $Q_m$ les matrices (diagonales) de ces formes sur les 
bases canoniques.

L'\ali $\varphi:E\rightarrow F $ donne lieu \`a une 
\ali $E'\rightarrow F'$ que nous notons encore 
$\varphi$ et qui est d\'efinie par la m\^{e}me matrice 
sur les bases canoniques. Il existe alors une unique
\ali  $\varphi\cir:F'\rightarrow E' $ v\'erifiant:
\begin{equation} \label{eqApAdjbis}
\forall x\in E'\quad \forall y\in F'\quad \quad 
\gen{\varphi(x),y 
}_{F'}^t=\gen{x,\varphi\cir(y)}_{E'}^{t}
\end{equation}
La matrice $A\cir$ de $\varphi\cir$ 
sur les bases canoniques est alors
\begin{equation} \label{eqDefAcirc}
A\cir\,=\, {Q_n}^{-1}\,\tra{\!A}\;Q_m\,,
\end{equation}
puisqu'on doit avoir pour tous 
$X\in\A(t)^{n{\times}1},\;Y\in\A(t)^{m{\times}1}$:
$ 
\tra{(A\,X)}\,Q_m\,Y\;=\tra{\!X}\,{Q_n}\,(A\cir\,Y).$

On v\'erifie que $(AB)\cir=B\cir \,A\cir$ et $(A\cir)\cir=A$.

En pratique si ${A}=(a_{i,j})$ on obtient
$A\cir=(t^{j-i}\,a_{j,i})$, par exemple:
$$ A\;=\;\cmatrix{
a_{11}&a_{12}&a_{13}&a_{14}&a_{15}\cr
a_{21}&a_{22}&a_{23}&a_{24}&a_{25}\cr
a_{31}&a_{32}&a_{33}&a_{34}&a_{35}
}
,\quad A\cir\;=\;\cmatrix{
a_{11}&t\,a_{21}&t^2\,a_{31}\cr
t^{-1}\,a_{12}&a_{22}&t\,a_{32}\cr
t^{-2}\,a_{13}&t^{-1}\,a_{23}&a_{33}\cr
t^{-3}\,a_{14}&t^{-2}\,a_{24}&t^{-1}\,a_{34}\cr
t^{-4}a_{15}&t^{-3}a_{25}&t^{-2}a_{35}
}
$$

\subsection{Id\'eaux de Gram  et \idds} 
\label{subsecIGID}

\begin{definition} 
\label{defIG} 
Soit  une matrice $A\in\A^{m\times n}$.
On d\'efinit les \pols de Laurent $\,\cG'_{k}(A)(t)=g_{k}(t)\in\A[t,1/t]$, 
 et  les {\em \coes de Gram \gnes de $A$},    
$\,\cG'_{k,\ell}(A)=g_{k,\ell}$ comme suit:
\begin{equation} \label{eqDefGrambis}
\left\{
\begin{array}{rcll} 
\rP{AA\cir}(Z)&  = & 1 \, + \, g_1(t) \, Z
\, + \, \cdots \, + \, g_m(t) \, Z^m     \\[2mm] 
g_{k}(t)&  = & t^{-k(n-k)} \,\left(\sum_{\ell=0}^{k(m+n-2k)} 
\,g_{k,\ell}\,t^\ell \right) 
\end{array}
\right.
\end{equation}
Autrement dit $g_k(t)=\rd k(AA\cir)$. Nous d\'efinissons aussi
$\cG'_0(A)=1$ et $\cG'_\ell(A)=0$ pour $\ell>m.$

\noindent Les \emph{\ids de Gram de la matrice $A$}  sont les id\'eaux 
$\cC_k(A)$ d\'efinis par: 

\ss\centerline{$\cC_k(A) \; 
:=  \;$  l'id\'eal   engendr\'e  par  les    $g_{h,\ell}$  pour tous les $h\geq 
k$.
}
\end{definition}

\begin{proposition} 
\label{propIG} 
On a $\sqrt{\cC_k(A)}=\sqrt{\cD_k(A)}$.
Plus pr\'ecis\'ement, avec un entier $r$ qui ne d\'epend que de $(m,n,k)$  on a:
 $$\cD_k(A)^r\subset \cC_k(A)\subset \cD_k(A)^2\subset \cD_k(A).$$
\end{proposition}
\begin{proof}{Preuve}
Les inclusions $\cC_k(A)\subset \cD_k(A)^2\subset \cD_k(A)$ sont claires. Dans le 
cas des corps on sait que $\, \cC_k(A)= \cD_k(A)$ (cette \egt est essentiellement 
une reformulation du r\'esultat de Mulmuley, cf. \cite{DiGL,Mul}).
On peut donc conclure par le Nullstellensatz formel qu'il existe
un entier  $r>0$ tel que
$\cD_k(A)^r\subset \cC_k(A)$.  Avoir un tel $r$ explicitement demande un peu plus 
de travail.
\junk{, par exemple faire une relecture ad\'equate de la preuve du lemme 
\ref{lemMul} qui est la vraie raison du lemme
\ref{lemGkbis}.} 
Nous nous en dispenserons car nous n'en aurons pas besoin pour nos calculs \gui{en 
temps polynomial}. 
\end{proof}

Notez que les \idds de $A\cir$ sont \'egaux \`a ceux de $A$.

\begin{corollary} 
\label{corpropIG} 
Si $A$ est \lnlz, les \ids de Gram de $A$ sont \'egaux \`a ses \idds et sont 
engendr\'es par des \idmsz. 
\end{corollary}
\begin{corollary} 
\label{corpropIGbis} 
R\'eciproquement:
\begin{enumerate}
\item Si $\cD_{k+1}(A)=0$ et $\cC_{k}(A)=\gen{1}$, 
 $A$ est \lnl de rang $k$.
\item Si $\A$ est r\'eduit (i.e. $0$ est le seul \elt nilpotent), si 
$\cC_{k+1}(A)=0$ et $\cC_{k}(A)=\gen{1}$, 
 $A$ est \lnl de rang $k$.
\item Si $\A$ est r\'eduit et si les \ids de Gram de $A$ sont engendr\'es par des 
\idmsz, alors $A$ est \lnlz.  
\item Si $\A$ est r\'eduit et a pour seuls \idms $0$ et $1$ (en particulier si 
$\A$  est int\`egre)
la matrice $A$ est \lnl \ssi il existe un entier $k$ v\'erifiant: 
$\cC_{k+1}(A)=0$ et $\cC_{k}(A)=\gen{1}$.
\end{enumerate}
\end{corollary}

Notez que la condition \gui{$\cC_{k+1}(A)=0$ et $\cC_{k}(A)=\gen{1}$}
revient \`a dire que $\rP{AA\cir}(Z)$ est de degr\'e $\leq k$ en $Z$ et que son 
\coe  $g_k(t)$ est inversible dans $\A(t)$.

\subsection{Calcul pratique d'un \ingz}
\label{subsecCalPra}

Rappelons que notre but est de donner un test rapide pour savoir si une matrice 
est \lnlz, et, en cas de r\'eponse positive, de calculer un \ing de la matrice.

\subsubsection*{Le cas du rang constant} 

C'est par exemple s\^{u}rement le cas si l'anneau n'a pas d'autre \idm que $0$ et 
$1$.

\begin{theorem} 
\label{thIGGRC2} 
Soit  une matrice $A\in\A^{m{\times}n}$.
On rappelle que $A\cir=Q_n^{-1}(t)\, A\, Q_m(t)$ et que
 $\rP{AA\cir}(Z)=\det(\I_n+ZAA\cir)=1+\sum_{1\leq \ell\leq n} g_\ell(t)\,Z^\ell$.
\Propeq
\begin{enumerate}
\item \label{thIGGRC2a} $A$ est \lnl de rang $k$ sur $\A$.
\item \label{thIGGRC2b} $A$ est \lnl de rang $k$ sur $\A(t)$. 
\item \label{thIGGRC2c} $A$ est \nl de rang $k$ sur $\A(t)$. 
\item \label{thIGGRC2e} $A$ est de rang $\leq k$ et 
le \pol $t^{k(n-k)}g_k(t)$ est primitif.  
\item \label{thIGGRC2f} $A$ et $A\cir$ sont \cros sur $\A(t),$ 
de rang $k$. 
\item \label{thIGGRC2g} $A$ et $A\cir$ sont \cros sur $\A(t)$,
 $\deg_Z(\rP{AA\cir})\leq k$  et
$t^{k(n-k)}g_k(t)$ est primitif. 
\item \label{thIGGRC2i} $\deg_Z(\rP{AA\cir})\leq k$, 
le \pol $t^{k(n-k)}g_k(t)$ est primitif et 
on a $A\, \Adj^{(k)}_{A\cir}(A) \, A=g_k(t)\,A$.
\end{enumerate}
Si $\A$ est un anneau r\'eduit, la condition \ref{thIGGRC2i} se simplifie en 
\gui{$\deg_Z(\rP{AA\cir})\leq k$ et 
le \pol $t^{k(n-k)}g_k(t)$ est primitif}.
Lorsque les conditions sont v\'erifi\'ees la matrice  
$\Adj^{(k)}_{A\cir}(A)/g_k(t)$ est l'\ing de
$A$ via $A\cir$ sur l'anneau $\A(t)$.
\end{theorem}
\begin{proof}{Preuve}
Le fait que \ref{thIGGRC2b} implique \ref{thIGGRC2c} a \'et\'e expliqu\'e dans la 
section \ref{subsecMrcLibre}.\\
De la caract\'erisation de \ref{thIGGRC2a} par le fait que $\cD_{k+1}(A)=0$  et  
$\cD_k(A)=\gen{1},$ on d\'eduit facilement l'\'equivalence de \ref{thIGGRC2a} et 
\ref{thIGGRC2b}.\\
Le corollaire \ref{corpropIG} montre que \ref{thIGGRC2a} implique que
le \pol $t^{k(n-k)}g_k(t)$ est primitif. En particulier \ref{thIGGRC2a} implique 
\ref{thIGGRC2e}. Il montre aussi l'\'equivalence de \ref{thIGGRC2g} et 
\ref{thIGGRC2f}.\\
Le \tho \ref{thCroiRgCst} montre que \ref{thIGGRC2e} implique \ref{thIGGRC2f}, 
lequel implique clairement  \ref{thIGGRC2b}.\\
On a donc l'\'equivalence des points \ref{thIGGRC2a} \`a \ref{thIGGRC2g}.\\
L'\'equivalence de \ref{thIGGRC2b} et \ref{thIGGRC2i} r\'esulte du \tho 
\ref{thGENCRAM1}. En effet ce \tho nous dit que 
$ A\, \Adj^{(k)}_{A\cir}(A) \, A\;\allowbreak \equiv\;\allowbreak g_k(t)\,A\; 
\mathrm{mod}\; \cD_{k+1}(A).$
Donc si $A$ est de rang $\leq k$ on a l'\egtz. Par ailleurs si on a l'\egtz,
$A$ est \lnl sur $\A(t)$ (condition \ref{e} dans le \tho \ref{theoremIFD}), et le 
rang est fourni par le 
corol\-laire~\ref{corpropIG}.\\
Le cas r\'eduit a d\'eja \'et\'e vu (proposition \ref{propIG} et corollaire 
\ref{corpropIGbis}). 
\end{proof}

D\'ecrivons maintenant un \algo \gui{rapide}
 pour savoir si une matrice est \lnl de rang constant, et, en cas de r\'eponse 
positive, pour calculer un \ing de la matrice.
Cet \algo fonctionne en utilisant 
la caract\'erisation \ref{thIGGRC2i} dans le \tho pr\'ec\'edent.

Comme nous nous int\'eressons pour le moment uniquement au rang constant, les 
seuls tests dont nous aurons besoin sont les suivants:
le test d'\'egalit\'e \`a $0$ dans $\A$ et le test \gui{$1\in\gen{x_1,\ldots 
,x_n}~?$}. 

On proc\`ede comme suit.
\begin{enumerate}
\item On calcule $A\cir$ d\'efini par l'\egt (\ref{eqDefAcirc}).

\item On calcule les \pols de Gram $g_k(t)$ d\'efinis par l'\egt
(\ref{eqDefGrambis}). 
Ceci se fait en calculant le \polcar de
$A\,A\cir$ si  $m\leq n$ ou celui de $A\cir\,A$ si $m>n$. 
\item On cherche la plus grande valeur de $k$ pour laquelle l'\id $\cC_{k}(A)$ est 
non nul. Pour cela on teste les \pols de Gram \gnes $g_{\ell}(t)$ pour des valeurs 
d\'ecroissantes de $\ell$ et on s'arr\^{e}te au premier non nul.  Pour le plus 
grand $k$ tel que $\cC_{k}(A)\neq 0$ 
on teste si $\cC_{k}(A)=\gen{1}$.  \\
Si la r\'eponse est n\'egative $A$ n'est pas \lnl de rang constant.
Si la r\'eponse est positive et si l'anneau est r\'eduit alors $\cD_{k+1}(A)$ est 
nul et la matrice est \lnl de rang $k$.
\item Dans tous les cas, si la r\'eponse est positive, on calcule la matrice 
$B=\Adj^{(k)}_{A\cir}(A)$ \`a \coes dans $\A[t,1/t]$ 
donn\'ee par l'\egt (\ref{eqDefAdjficirc2}) page \pageref{eqDefAdjficirc2}.
Ensuite on teste si:
\begin{equation} \label{eqTestCrucial}
A\, B(t)\, A=g_k(t)\,A
\end{equation}
(ce test est inutile si l'anneau est r\'eduit).
Nous savons d\'ej\`a (\tho \ref{thGENCRAM1}) que
$$ A\, B(t)\, A\;\equiv\;g_k(t)\,A\quad \mathrm{mod}\quad \cD_{k+1}(A)
$$
En cas de r\'eponse n\'egative, $\cD_{k+1}(A)\neq 0$ et
  $A$ n'est pas \lnlz. En cas de r\'eponse positive  $A$ est  \lnlz,
au moins sur l'anneau $\A(t)$ car $t^{k(n-k)}\,g_k(t)$ est un \pol primitif.

\item Il nous reste \`a calculer un inverse \gne de $A$ \`a \coes dans $\A$.
L'\egt (\ref{eqTestCrucial}) peut \^{e}tre lue en chaque degr\'e  $t^\ell$ (avec 
$-k(n-k)\leq \ell\leq k(m-k)$) comme une \egt dans $\A^{m{\times}n}$:
$A\, B_\ell\, A=g_{k,\ell}\,A$.
Comme on conna\^{\i}t une \coli $\sum_\ell \alpha_{\ell}\,g_{k,\ell}$
des \coes de $g_k(t)$ qui est \'egale \`a $1$, la \coli correspondante des \egts 
en chaque degr\'e $\ell$ nous donne la matrice $B'=\sum_\ell 
\alpha_{\ell}\,B_{\ell}$ qui v\'erifie  $A\, B'\, A = A$.
\end{enumerate}

Ici il semble peu probable que l'on ait aussi   $B'\, A\, B' = B'$, sauf si les 
$\alpha_\ell$ sont obtenus en sp\'ecialisant $t$ (ce qui peut se faire si 
$g_k(\tau)$ est inversible pour une valeur particuli\`ere de $\tau$). De toute 
fa\c{c}on, on peut toujours remplacer $B'$ par $B'\, A\, B'$ pour avoir un vrai 
\ingz.

\medskip Voyons maintenant la complexit\'e de cet \algoz.

Nous n'utilisons ni la multiplication rapide des matrices, ni celle
des \polsz, qui, naturellement, am\'elioreraient de fa\c{c}on substantielle les
bornes calcul\'ees.

Voici notre calcul des bornes, \'etape par \'etape.
\begin{enumerate}
\item Co\^{u}t n\'egligeable.
\item 
Posons $p=\inf(m,n)$. 
 Le calcul du \polcar consomme $\cO(p^4)$ \oparis dans $\A[t]$ 
portant sur des \pols de degr\'e $\leq p\,(n+m)$ et donc
 $\cO(p^6\,(n+m)^2)$  \oparis dans~$\A$.
\item Cette \'etape consomme $\cO((k\,(n+m-2k))^s)$ \ops \elrsz. 
\item Le nombre d'\oparis est en  $\cO(p^5\,(n+m)^2)$, et le test consomme 
 $\cO((p^3\,(n+m))$ op\'erations \elrsz.
\item Nombre d'\oparis n\'egligeable par rapport aux \'etapes 2 ou 4.
\end{enumerate}

R\'esumons.

\begin{theorem} 
\label{thIGGRC} 
Soit $\A$ un anneau avec test d'\'egalit\'e \`a $0$ et test 
\gui{$1\in\gen{x_1,\ldots ,x_n}~?$}. 
On peut tester si une matrice $A\in\A^{m{\times}n}$ est \lnl
de rang constant, et en cas de r\'eponse positive, calculer un inverse \gne de la 
matrice. Soit $p=\min(m,n)$, $q=\max(m,n)$.
Si le premier test consomme une \op \elr
et si le deuxi\`eme consomme un nombre
d'\ops \elrs  en $\cO(n^s)$, ces calculs consomment $\cO(p^6\,q^{2})$ \oparis et 
$\cO(p^3\,q+q^{2s})$ autres \ops \elrsz.
Avec les m\^{e}mes bornes de complexit\'e, on calcule un \ing de $A$ et des 
matrices de
\prn sur le noyau et sur l'image de $A$.
\end{theorem}

D\`es que $p\geq 2$ ces bornes sont, pour $q$ assez grand, bien meilleures que 
celles obtenues si on ex\'ecute (na\"{\i}vement) un
\algo qui calcule  tous les mineurs de la matrice.

\subsubsection*{Le cas g\'en\'eral} 

Le cas g\'en\'eral se ram\`ene au cas pr\'ec\'edent: cf. lemme
\ref{lemIdIdm} et section \ref{subsecGenRangNonConst}.

On obtient donc.

\begin{theorem} 
\label{thIGGRC2bis} 
Soit  une matrice $A\in\A^{m{\times}n}$.
Les propri\'et\'es suivantes sont \'equivalentes:
\begin{enumerate}
\item $A$ est \lnl sur $\A$. 
\item $A$ est \lnl sur $\A(t)$. 
\item $A$ est \qnl sur $\A(t)$. 
\item $A$ et $A\cir$ sont \cros sur $\A(t)$.  
\end{enumerate}
\end{theorem}

Signalons que le lemme \ref{lemIdIdm} peut \^{e}tre am\'elior\'e en raison du fait 
suivant: le produit de deux \itfs localement principaux donn\'es respectivement
par $n$ et $m$ \gtrs est un \itf donn\'e par $n+m-1$ \gtrs (voir par exemple
\cite{dlqs}). Ceci nous permet d'obtenir la  complexit\'e suivante.

\begin{theorem} 
\label{thIGGG} 
Sur un anneau $\A$ fortement discret,
on peut tester si une matrice $A\in\A^{m{\times}n}$ est \lnlz, et en cas de 
r\'eponse positive, calculer un inverse \gne
de la matrice.  Soit $p=\min(m,n)$, $q=\max(m,n)$.
Si l'anneau est  $\cO(n^{s})$-fortement discret, ces calculs consomment 
$\cO(p^6\,q^{2}+p\,q^{4})$ \oparis et 
$\cO(p^4\,q+pq^{2s+1})$ autres \ops \elrsz.
Avec les m\^{e}mes bornes de complexit\'e, on calcule un \ing de $A$ et des 
matrices de
\prn sur le noyau et sur l'image de $A$.
\end{theorem}

\subsection{Les statisticiens indiens}  

Les num\'ericiens puis les statisticiens ont d\'evelopp\'e une th\'eorie
des \gui{\ingsz} d'abord pour le cas des corps $\RR$ et $\CC$ 
mais ensuite pour des anneaux commutatifs arbitraires. Cette th\'eorie est 
essentiellement l'\'equivalent des \thos \ref{theoremIFD} et
\ref{propPTFDec},  avec des pr\'eoccupations particuli\`eres de calculs explicites 
et de formules pr\'ecises. Ce sont les statisticiens indiens qui ont d\'evelopp\'e 
le plus cette th\'eorie.

En fait cette convergence n'est pas fortuite. A la base, il y a le fait que 
num\'eriquement, $\RR$ et $\CC$  ne se comportent 
\gui{pas vraiment} comme des  corps, \`a cause de la difficult\'e du test \`a 
z\'ero  (voire son impossibilit\'e) qui est la source de ph\'enom\`enes 
d'instabilit\'e, li\'es par exemple au calcul de l'inverse d'un nombre trop proche 
de 0.

Dans \cite{PB} la formule  (\ref{eqIg3}) est \'etablie pour le cas suivant: la 
matrice de $\varphi\bul$ est de la forme $M\tra{A^\star}N$ 
o\`u $x\mapsto x^\star$ est un \auto involutif de l'anneau $\A$, suppos\'e 
int\`egre. 
Nous n'avons pas trouv\'e la formule (\ref{eqIg3}) elle-m\^{e}me dans le cas le 
plus g\'en\'eral, mais cela ne signifie pas qu'elle n'existe pas dans la 
litt\'erature \gui{indienne}.
Il y a deux livres de r\'ef\'erence pour ces \'ecrits: \cite{Bha} et \cite{RM}. 

De mani\`ere surprenante, nous avons rarement trouv\'e chez les statisticiens 
indiens 
de formules analogues \`a l'\'equation (\ref{eqIg2}) du \tho \ref{thIgRangr} (il y 
en a une dans \cite{RM} dans un cas particulier) mais plut\^ot
des formules du style  (\ref{eqDefAdjficirc}) et (\ref{eqIg3}).


\newpage 
\tableofcontents

\end{document}